\documentclass[12pt]{article}
\usepackage[english]{babel}
\usepackage{amsmath,amsthm,amsfonts,amssymb,epsfig,vector}
\usepackage[left=1in,top=1in,right=1in]{geometry}

\newtheorem{thm}{Theorem}[section]
\newtheorem{lem}[thm]{Lemma}
\newtheorem{prop}[thm]{Proposition}
\newtheorem{cor}[thm]{Corollary}
\newtheorem{df}[thm]{Definition}
\newtheorem{rem}{Remark}

\newcommand{\E}{\mathbb{E}}
\newcommand{\G}{\mathcal{G}}

\newcommand{\R}{\mathbb{R}}
\newcommand{\Z}{\mathbb{Z}}
\newcommand{\N}{\mathbb{N}}
\newcommand{\F}{\mathcal{F}}

\newcommand{\e}{\varepsilon}

\begin{document}

\date{October 30, 2011}
\title{On the Number of Ground States of the Edwards-Anderson Spin Glass Model}
\author{\Large Louis-Pierre Arguin  \thanks{L.-P. Arguin was supported by the NSF grant DMS-0604869 during part of this work.} \\ \small Universit\'e de Montr\'eal \\ \small Montr\'eal, QC H3T 1J4 Canada
\and \Large Michael Damron\thanks{M. Damron is supported by an NSF postdoctoral fellowship.} \\ \small Princeton University \\ \small Princeton, NJ 08544, USA}
 
\maketitle

\begin{abstract}
Ground states of the Edwards-Anderson (EA) spin glass model are studied on infinite graphs with finite degree.
Ground states are spin configurations that locally minimize the EA Hamiltonian on each finite set of vertices.
A problem with far-reaching consequences in mathematics and physics is to determine the number of ground states for the model on $\Z^d$ for any $d$.
This problem can be seen as the spin glass version of determining the number of infinite geodesics in first-passage percolation
or the number of ground states in the disordered ferromagnet.
It was recently shown by Newman, Stein and the two authors that, on the half-plane $\Z\times\N$, there is a unique ground state
(up to global flip) arising from the weak limit of finite-volume ground states for a particular choice of boundary conditions.
In this paper, we study the entire set of ground states on the infinite graph, proving that the number of ground states on the half-plane must
be two (related by a global flip) or infinity. This is the first result on the entire set of ground states in a non-trivial dimension. In the first part of the paper, we develop tools of interest to prove the analogous result on $\Z^d$.
\end{abstract}

\tableofcontents

\section{Introduction}

\subsection{The model and the main result}

We study the Edwards-Anderson (EA) spin glass model on an infinite graph $G=(V,E)$ of finite degree.
We mostly take $G=\Z^d$ (and further, $d=2$),
and $G=\Z\times\N$, a half-plane of $\Z^2$.

For a finite set $A \subseteq V$, consider the set of spin configurations $\Sigma_A = \{-1,+1\}^A$ and for $\sigma \in \Sigma_A$, the Hamiltonian (with free boundary conditions)
\begin{equation}
\label{eqn: H}
H_{J,A}(\sigma) = -\sum_{\substack{\{x,y\}\in E\\ x,y \in A}} J_{xy}\sigma_x\sigma_y\ ,
\end{equation}
where the $J_{xy}$'s (the {\it couplings}) are taken from an i.i.d. product measure $\nu$. We assume that the distribution of each $J_{xy}$ is continuous with support equal to $\mathbb{R}$. For inverse temperature $\beta>0$ the Gibbs measure for $A$ is 
\[
\text{{\bf G}}_{J,A}(\sigma) = \frac{1}{Z_{J,A}} \exp(-\beta H_{J,A}(\sigma)) \ , \ Z_{J,A}=\sum_{\sigma\in\Sigma_A} \exp(-\beta H_{J,A}(\sigma))\ .
\]
As temperature approaches 0 ($\beta \to \infty$) the Gibbs measure converges weakly to a sum of two delta masses, supported on the spin configurations with minimal value of $H_{J,A}$. 
These spin configurations (related by global flip) can be seen to be characterized by the following local flip property: for each $B \subseteq A$, we have
\[
\sum_{\substack{\{x,y\} \in \partial B \\ x,y \in A}} J_{xy}\sigma_x\sigma_y \geq 0\ .
\]
Here the set $\partial B\subset E$ is defined as all edges $\{x,y\}$ such that $x \in B$ and $y \notin B$. 
The advantage is that this definition makes sense for infinite sets $A$. 
For this reason, we define the {\it set of ground states on the infinite graph $G$ at couplings $J$} by
\begin{equation}\label{eq:GSproperty}
\mathcal{G}(J)=\{\sigma\in \{-1,+1\}^{V}: \forall A\subset V \text{ finite, }\sum_{\{x,y\}\in\partial A}J_{xy}\sigma_x\sigma_y\geq 0\}\ .
\end{equation}
In other words, elements of $\mathcal{G}(J)$ are the spin configurations minimizing the Hamiltonian locally for the coupling realization $J$.
Clearly, $\sigma\in \mathcal{G}(J)$ if and only if $-\sigma\in \mathcal{G}(J)$.
The goal of this paper is not to determine precisely the cardinality of $\mathcal{G}(J)$ but rather
to rule out possibilities other than two or infinity.
Our main result is to prove such a claim in the case of the EA model on the two-dimensional
half-plane.
\begin{thm}\label{thm: mainthm}
For the EA model on the half-plane $G=\Z\times \N$,
the number of ground states $|\mathcal{G}(J)|$ is either $2$ with $\nu$-probability one or $\infty$ with $\nu$-probability one.
\end{thm}

\subsection{Previous results}

A main question in the theory of short-range spin glasses is to understand the structure of the set $\mathcal{G}(J)$,
and in particular its cardinality.
This problem is the zero-temperature equivalent of understanding the structure and the cardinality of the
set of pure states, the set of infinite-volume Gibbs measures of the EA model that are extremal. 
It is easy to check that for $G=\Z$, $\mathcal{G}(J)$ has only two elements: the flip-related
configurations $\sigma$ defined by the identity $\sigma_x\sigma_y=\text{sgn} J_{xy}$.
However, it is not known how many elements are in $\mathcal{G}(J)$ for $G=\Z^d$ when $d>1$.
(We will see in the next section that the cardinality of $\mathcal{G}(J)$ must be a constant number $\nu$-almost surely.)
It is expected that $|\mathcal{G}(J)|=2$ for $d=2$ \cite{Middleton99,PY99}
(see also \cite{Loebl04} for a possible counterargument to this).
There are competing predictions for higher dimensions.
The Replica Symmetry Breaking (RSB) scenario would predict $|\mathcal{G}(J)|=\infty$ for $d$ high enough,
and the droplet/scaling proposal would be consistent with $|\mathcal{G}(J)|=2$ in every dimension.
We refer to \cite{N97,NS03} for a detailed discussion on ground states of disordered systems 
or pure states at positive temperature.

There have been several works on ground states of the EA model in the physics and mathematics literature;
a partial list includes \cite{ADNS10, FH86, Loebl04, Middleton99, N97, NS01, NS03, PY99}.
The present work appears to give the first rigorous result about the entire set of ground states $\mathcal{G}(J)$. 
Previous rigorous results have focused on the so-called {\it metastates on ground states}.
A metastate is a $J$-dependent probability measure on $\{-1,+1\}^V$ supported on ground states.
It is constructed using a sequence of finite graphs $G_n$ converging to $G$.
For a given realization $J$ and $n$, the ground state on $G_n$ is unique up to a global flip.
We identify the flip-related configurations and write $\sigma_n^*(J)$ for them.
A metastate is obtained by considering a converging subsequence of the 
measures $\big(\nu(dJ) ~ \delta_{\sigma_n^*(J)}\big)_n$,
where $\delta_{\sigma_n^*(J)}$ is the delta measure on the ground state of $G_n$
for the coupling realization $J$. If $\kappa$ denotes a subsequential limiting measure, then
sampling from $\kappa$ gives a pair $(J,\sigma)\in  \R^E \times \{-1,+1\}^V$. 
A metastate is the conditional measure $\kappa$ given $J$ and is denoted $\kappa_J$.
It is not hard to verify that $\kappa_J$ is supported on $\mathcal{G}(J)$.

It was proved in \cite{ADNS10} that the ground state of the EA model on the half-plane with horizontal periodic boundary conditions 
and free boundary condition at the bottom is unique in the metastate sense.
Precisely, for a sequence of boxes $G_n$ that converges to the half-plane, the
limit $\kappa_J$ produced by the metastate construction is unique and is given by
a delta measure on two flip-related ground states. Though the metastate construction is very natural, 
it is important to stress that the measure thus obtained is not necessarily supported on the whole set $\mathcal{G}(J)$.
It may be that some elements of $\mathcal{G}(J)$ do not appear in the support of the metastate, due to the choice of boundary
conditions on $G_n$ or to the fact that the subsequence in the metastate construction is chosen independently of $J$.
Therefore uniqueness in the metastate sense does not answer the more general question of the number of ground states.

It is natural from a statistical physics perspective to study the set $\mathcal{G}(J)$ by looking at probability measures on it.
One challenge is to construct
probability measures on $\mathcal{G}(J)$ that have a nice dependence on $J$,
namely measurability and translation covariance.
The metastate (with suitably chosen boundary conditions) briefly described above is one such measure. 
The main idea of the present paper is to consider another measure, the {\it uniform measure on $\mathcal{G}(J)$}
$$
\mu_J:= \frac{1}{|\mathcal{G}(J)|}\sum_{\sigma\in \mathcal{G}(J)} \delta_{\sigma}\ .
$$
For $\mu_J$ to be well-defined it is necessary to assume that $|\mathcal{G}(J)|$ is finite. 
Like the metastate, the uniform measure on ground states depends nicely on $J$:
see Proposition~\ref{prop: mu mble} and Lemma~\ref{lem: covariance}.
The strategy to prove a ``two-or-infinity" result is to assume that $|\mathcal{G}(J)|<\infty$ and to conclude that
it implies that $\mu_J$ is supported on two spin configurations related by a global flip (that is, $|\mathcal{G}(J)| = 2$).
The approach is similar to the proof of uniqueness in \cite{ADNS10} using the  interface between ground states,
though new tools need to be developed.
For spin configurations $\sigma$ and $\sigma'$, define the {\it interface} $\sigma \Delta \sigma'$ as 
$$
\sigma \Delta \sigma'=\{\{x,y\}\in E:\sigma_x\sigma_y\neq \sigma'_x\sigma'_y \}\ .
$$
It will be shown for the half-plane that
$$
\int \nu(dJ) ~ \mu_J\times\mu_J\{(\sigma,\sigma'):\sigma\Delta \sigma'=\emptyset\}=1 \ .
$$
This implies that $\mu_J$ is supported on two flip-related configurations for $\nu$-almost all $J$ since
$\sigma\Delta\sigma'=\emptyset$ if and only if $\sigma=\sigma'$ or $\sigma=-\sigma'$.

Before going into the details of the proofs, we remark that the problem of determining the number of ground states for the EA model can be seen as a spin-glass version of a first-passage
percolation problem. Indeed, one question in two-dimensional first-passage percolation is to determine whether there exist infinite geodesics. 
These are doubly-infinite curves that locally minimize the sum of the random weights between vertices of the graph.
This problem is equivalent to determining whether there exist more than two flip-related ground states in the (ferromagnetic) Ising model with random couplings. 
The Hamiltonian of the ferromagnetic model is the same as in \eqref{eqn: H}, but the distribution of $J$ is restricted to the positive half-line.
The reader is referred to \cite{Wehr97} for the details of the correspondence. 
It was proved by Wehr in \cite{Wehr97} that the number of ground states for this model is either two or infinity in dimensions greater or equal to two.
On the half-plane, it was shown by Wehr and Woo \cite{WehrWoo98} that the number of ground states is two.
Contrary to the ferromagnetic case, the study of ground states of the EA spin glass model presents technical difficulties that stem
from the presence of positive and negative couplings. This feature rules out monotonicity of the partial sums of couplings along an interface.

The paper is organized into two main parts as follows.
The first part develops general tools to study ground states of the EA model.
Precisely, in Section 2, elementary properties of the set $\mathcal{G}(J)$ are derived for general graphs.
In particular, the dependence of $\mathcal{G}(J)$ on a single coupling is studied.
Properties of probability measures on $\mathcal{G}(J)$ are investigated in Section 3 with
an emphasis on the uniform measure on $\mathcal{G}(J)$.
The second part of the paper consists of the proof of Theorem~\ref{thm: mainthm} and is contained in Section 4.\\

\noindent {\bf Acknowledgements}
Both authors are indebted to Charles Newman and Daniel Stein for having introduced them to the subject of short-range spin glasses and for numerous discussions on 
related problems. L.-P. Arguin thanks also Janek Wehr for discussions on the problem of the number of ground states in spin glasses and in disordered ferromagnets.

 \section{Elementary properties of the set of ground states}\label{sec:gs}
In this section, unless otherwise stated, we consider the EA model on a connected graph $G=(E,V)$ of finite degree.
We assume that there exists a sequence of subgraphs $(G_n)$ that converges locally to $G$.
Throughout the paper, we will use the following notation:
$\Omega_1 = \R^{E}$, and ${\cal F}_1$ is the Borel sigma-algebra generated by its product topology;
$\Omega_2 = \{-1,+1\}^{V}$ and ${\cal F}_2$ is the corresponding product sigma-algebra.

\subsection{Measurability}
We first note that the set of ground states is compact.
\begin{lem}
$\mathcal{G}(J)$ is a non-empty compact subset of $\Omega_2$ (in the product topology) for all $J$.
In particular, the set of probability measures on $\mathcal{G}(J)$ is compact in the weak-* topology on the set of probability measures on $\Omega_2$.
\end{lem}
\begin{proof}
The fact that $\mathcal{G}(J)$ is non-empty follows by a standard compactness argument, taking a subsequence of ground states for the Hamiltonian
\eqref{eqn: H} with $A=G_n$.
The function $\sigma\mapsto \sum_{\{x,y\}\in\partial A}J_{xy}\sigma_x\sigma_y$ is continuous in the product topology for a given finite $A$ and $J$.
Therefore, the set $\{\sigma\in \{-1,+1\}^{V}: \sum_{\{x,y\}\in\partial A}J_{xy}\sigma_x\sigma_y\geq 0\}$ is closed. 
Since $\mathcal{G}(J)$ is the intersection of these sets over all finite $A$ by \eqref{eq:GSproperty}, it is closed. Being a closed subset of the compact space $\Omega_2$, it is also compact.
The second statement of the lemma follows from the first.
\end{proof}

The next result is necessary for the uniform measure to be well-behaved and  to later apply the ergodic theorem to $|\mathcal{G}(J)|$.
\begin{prop}\label{prop:mble}
The random variable $J\mapsto |\mathcal{G}(J)|$ is $\mathcal{F}_1$-measurable.
\end{prop}

\begin{proof}
Consider a sequence of finite graphs $\Lambda_n\subset G$, a configuration $\sigma_n$ on $\Lambda_n$ and a configuration $\bar \sigma_n$ on the external boundary of $\Lambda_n$ (that is, all vertices that are not in $\Lambda_n$ but are adjacent to vertices in it). The condition that $\sigma_n$ is a ground state in $\Lambda_n$ with boundary conditions $\bar \sigma_n$ is a finite list of conditions of the form
\begin{equation}\label{eq:loopcondition}
\sum_{\{ x,y \} \in S} J_{xy} (\sigma_n)_x (\sigma_n)_y \geq 0 \mbox{ or } \sum_{\{ x,y \} \in S} J_{xy} (\sigma_n)_x (\bar \sigma_n)_y \geq 0
\end{equation}
for specific finite sets $S$ of edges. For any given $S$, the set of $J \in \Omega_1$ such that condition \eqref{eq:loopcondition} holds for fixed $\sigma_n$ and $\bar \sigma_n$ is then measurable (that is, it is in $\mathcal{F}_1$). Intersecting over all relevant sets $S$, we see that the following set is measurable:
\[
{\cal J}(\sigma_n,\bar \sigma_n) := \{ J ~:~ \sigma_n \mbox{ is a ground state in } \Lambda_n \mbox{ for the boundary condition } \bar \sigma_n\} \ .
\]

Next take $m < n$ and fixed configurations $\sigma_m$ on $\Lambda_m$, $\sigma_{m,n}$ on $\Lambda_n \setminus \Lambda_m$ and $\bar \sigma_n$ on the boundary of $\Lambda_n$. By a similar argument to the one given above, the set ${\cal J}(\sigma_m,\sigma_{m,n},\bar \sigma_n)$ of $J$ such that the concatenation of $\sigma_m$ and $\sigma_{m,n}$ is a ground state on $\Lambda_n$ with boundary condition $\bar \sigma_n$ is measurable. Taking the union over all $\sigma_{m,n}$ and $\bar \sigma_n$ for a fixed $\sigma_m$, we get that for $m<n$ and $\sigma_m$ fixed, the following set is measurable:
\[
{\cal J}(\sigma_m,n) := \{J ~:~ \mbox{ there is a ground state in } \Lambda_n \mbox{ (for some $\bar \sigma_n$)  that equals } \sigma_m \mbox{ on } \Lambda_m\}\ .
\]

If there exists a sequence of (possibly $J$-dependent) configurations $(\bar \sigma_n)$ such that there are ground states $(\sigma_n)$ on $\Lambda_n$ with boundary condition $\bar \sigma_n$ that converge to $\sigma$, then $\sigma$ is in $\mathcal{G}(J)$. Conversely, if $\sigma\in \mathcal{G}(J)$, such a sequence $(\sigma_n)$ exists by taking $\bar{\sigma}_n$ to be the restriction of $\sigma$ to the boundary. It follows that $\cap_{n\geq m} {\cal J}(\sigma_m,n)$ is the event that there is an infinite-volume ground state $\sigma$ for couplings $J$ that equals $\sigma_m$ on $\Lambda_m$. This event is thus measurable.

For fixed $m$ and a configuration $\sigma_m$ on $\Lambda_m$, let $F_{\sigma_m}(J)$ be the indicator of the event that there is an infinite-volume ground state $\sigma$ for couplings $J$ equal to $\sigma_m$ on $\Lambda_m$. By the above, it is $\F_1$-measurable. The proposition will then be proved once we show:
\begin{equation}\label{eq:numbergs}
|\mathcal{G}(J)|=\sup_m \sum_{\sigma_m} F_{\sigma_m}(J)\ .
\end{equation}
Here the sum is over all $\sigma_m$ on $\Lambda_m$. 
For any $m$, the sum $\sum_{\sigma_m} F_{\sigma_m}(J)$ equals the number of different $\sigma_m$'s that are equal to restrictions on $\Lambda_m$ of elements
of $\mathcal{G}(J)$ . So for each $m$,
\[
\sum_{\sigma_m} F_{\sigma_m}(J) \leq |\mathcal{G}(J)| 
\]
and the right side of \eqref{eq:numbergs} is at most $|\mathcal{G}(J)|$. To show equality in \eqref{eq:numbergs}, suppose first that $|\mathcal{G}(J)|$ is finite. We can choose $n$ so that the restriction 
 to $\Lambda_n$ of each element of $\mathcal{G}(J)$ is different. For this $n$, $\sum_{\sigma_n} F_{\sigma_n}(J) = |\mathcal{G}(J)|$ and \eqref{eq:numbergs} is established.
 If $|\mathcal{G}(J)|=\infty$, then for any $k\in\N$, we can find $n_k$ such that $\sum_{\sigma_{n_k}} F_{\sigma_{n_k}}(J)\geq k$. This is because we can take $\Lambda_n$ large enough
 so that there are at least $k$ elements of $\mathcal{G}(J)$ that are distinct on $\Lambda_n$. Taking the supremum over $k$ completes the proof of \eqref{eq:numbergs}.
\end{proof}

In the case $G=\Z^d$, it is easy to see that for any translation $T_a$ by a vector $a \in \Z^d$,
$|\mathcal{G}(J)|=|\mathcal{G}(T_a J)|$ where $(T_a J)_{xy} = J_{T_a(x)T_a(y)}$.
The ergodic theorem then implies that the random variable $|\mathcal{G}(J)|$ is constant
$\nu$-almost surely. The same holds when $G$ is the half-plane by considering only 
horizontal translations.

\begin{cor}
\label{cor: constant}
For $G=\Z^d$ or $G=\Z\times \N$, the number of ground states $|\mathcal{G}(J)|$
is a constant $\nu$-almost surely.
\end{cor}

The next result shows that if $|\mathcal{G}(J)|<\infty$ then the uniform measure $\mu_J$ is a random variable over $\mathcal{F}_1$.

\begin{prop}
\label{prop: mu mble}
Let $B \in {\cal F}_2$ and assume that $|\mathcal{G}(J)|<\infty$. The map
\[
J \mapsto \mu_J(B)
\]
is $\mathcal{F}_1$-measurable. Similarly, if $B'$ is a Borel set in $\Omega_2\times \Omega_2$, then the map
$J \mapsto \mu_J\times \mu_J(B')$ is  $\mathcal{F}_1$-measurable.
\end{prop}

\begin{proof}
By a standard approximation, it is sufficient
to prove the statement for $B$ of the form
\[
B = \{ \sigma~:~ \sigma = s_A \mbox{ on } A\}
\]
for some finite set $A$ and fixed configuration $s_A$ on $A$. 

Take a sequence of finite graphs $\Lambda_n$ converging to $G$.
We define
\[
F_{s_A}(J) =  \mbox{ number of } \sigma\in \mathcal{G}(J) \mbox{'s that equal } s_A \mbox{ on } A\ .
\]
Note that $\mu_J(B)$ is simply $F_{s_A}(J)$ divided by $|\mathcal{G}(J)|$. The variable $\mathcal{G}(J)$ is $\F_1$-measurable by Proposition \ref{prop:mble}. 
Thus it remains to show that $F_{s_A}(J)$ is also.

Exactly as in the last proof, if $n$ is so large that $\Lambda_n$ contains $A$ and if $s_{A,n}$ is any fixed spin configuration on $\Lambda_n \setminus A$, then the set ${\cal J}(s_A,s_{A,n})$ of all $J$ such that there is an element of $\mathcal{G}(J)$ that (a) equals $s_A$ on $A$ and (b) equals $s_{A,n}$ on $\Lambda_n \setminus A$ is measurable. Let $F_{s_A,s_{A,n}}(J)$ be the indicator of the event ${\cal J}(s_A,s_{A,n})$ and consider the random variable
\[
\sup_n \sum_{s_{A,n}} F_{s_A,s_{A,n}}(J)\ .
\]
Here the supremum is over all $n$ such that $A \subseteq \Lambda_n$. 
The same reasoning to prove \eqref{eq:numbergs} shows that $F_{s_A}(J)$ is equal to the above and is thus measurable. This completes the proof of the first claim.
The second assertion is implied by the first one since by a standard approximation, any measurable function on $\Omega_2\times\Omega_2$ can be approximated by linear combinations of indicator functions of sets of the form
$$
B_A:=\{(\sigma,\sigma'): \sigma=s_A \text{ on $A$ }, \sigma'=s_{A'} \text{ on $A'$ }\}
$$
for two finite sets $A$ and $A'$ of $G$. Since $\mu_J\times \mu_J(B_A)$ is equal to the product of the $\mu_J$-probability of each coordinate, measurability follows from the first part of the proposition.
\end{proof}

\subsection{Properties of the set of ground states}
In this section, we establish some elementary properties of the dependence of the set of ground states $\mathcal{G}(J)$ 
on a finite number of couplings.

Fix an edge $e=\{x,y\}$. We will sometimes abuse notation and write for simplicity
$$
J_e:=J_{xy} ~~\text{and}~~ \sigma_e:=\sigma_x\sigma_y \ .
$$
We are interested in studying how $\mathcal{G}(J)$ varies when $J_e$ is modified.
For simplicity, we will fix all other couplings and write $\mathcal{G}(J_e)$ for the set of ground states to stress the dependence on $J_e$.
From the definition \eqref{eq:GSproperty}, it is easy to see that if $\sigma\in \mathcal{G}(J_e)$ and $\sigma_e=+1$, then
$\sigma$ remains a ground state for coupling values greater than $J_e$. More generally:
\begin{lem}
\label{lem: monotone}
Fix an edge $e=\{x,y\}$. If $J_e\leq J_e'$ then
$$
\begin{aligned}
&\mathcal{G}(J_e)\cap \{\sigma: \sigma_e=+1\}\subseteq \mathcal{G}(J'_e)\cap \{\sigma: \sigma_e=+1\}\\
&\mathcal{G}(J_e)\cap \{\sigma: \sigma_e=-1\}\supseteq \mathcal{G}(J'_e)\cap \{\sigma: \sigma_e=-1\}
\end{aligned}
$$
\end{lem}
In view of the above monotonicity of the set of ground states, it is natural to introduce 
 the {\it critical value} of $\sigma\in \mathcal{G}(J_e)$ at $e$. Namely,
we define the critical value $C_e$ as
$$
\begin{aligned}
C_e(J,\sigma):=\begin{cases}
\inf\{J_e: \sigma\in \mathcal{G}(J_e)\} \text{ if $\sigma_e=+1$;}\\
\sup\{J_e: \sigma\in \mathcal{G}(J_e)\}\text{ if $\sigma_e=-1$.}
\end{cases}
\end{aligned}
$$
For future reference, we remark that from the definition,
\begin{equation}
\label{eqn: equiv}
\begin{aligned}
&\sigma\in \mathcal{G}(J_e)\text{ and }\sigma_e=+1  \Longrightarrow J_e\geq C_e(J,\sigma)\\
&\sigma\in \mathcal{G}(J_e)\text{ and }\sigma_e=-1 \Longrightarrow J_e\leq C_e(J,\sigma)\ .
\end{aligned}
\end{equation}

An elementary correspondence exists between the critical values and the energy required to flip finite sets of spins.
\begin{lem}
\label{lem: critical formula}
Let $\sigma\in \mathcal{G}(J_e)$. Then
\begin{equation}
\label{eqn: c_e}
\sigma_eC_e(J,\sigma)=-\inf_{A~:~e \in \partial A}\sum_{\substack{\{z,w\}\in \partial A\\ \{z,w\}\neq e}} J_{zw}\sigma_z\sigma_w\ ,
\end{equation}

In particular, for a given $\sigma$, $C_e(J,\sigma)$ does not depend on $J_e$.
\end{lem}
In this section, we will often omit the dependence on $J$ in the notation and write $C_e(\sigma)$ for simplicity.
From the above result, we see that this notation
is consistent with the fact that all couplings other than $J_e$ are fixed in this section.
\begin{proof}
The independence assertion is straightforward from the expression.
We prove the equation in the case of $\sigma_e=+1$. The other case is similar. 
Let $-\widetilde{C}_e(\sigma)$ be the right side of \eqref{eqn: c_e}.
If $C_e(\sigma)+\widetilde C_e(\sigma)>0$,
 there exists $\delta>0$ such that
$$
C_e(\sigma)-\delta+ \inf_{A: e\in \partial A}\sum_{\substack{\{z,w\}\in \partial A\\ \{z,w\}\neq e}} J_{zw}\sigma_z\sigma_w>0\ .
$$
In particular, $\sigma\in \mathcal{G}(J'_e)$ for $J'_e=C_e(\sigma)-\delta$, contradicting $C_e(\sigma)$ as the infimum of such values.
On the other hand if $C_e(\sigma)+\widetilde C_e(\sigma)<0$, there must exist a finite set $A$ such that
$$
C_e(\sigma)+\sum_{\substack{\{z,w\}\in \partial A\\ \{z,w\}\neq e}} J_{zw}\sigma_z\sigma_w<0\ .
$$
In particular this would hold for $C_e(\sigma)$ replaced by some $J_e>C_e(\sigma)$, contradicting 
the definition of $C_e(\sigma)$, because we should have $\sigma\in \mathcal{G}(J_e)$ for all $J_e>C_e(\sigma)$.
\end{proof}

The distance $|J_e-C_e(\sigma)|$ from $J_e$ to the critical value is called the {\it flexibility of $e$} and is denoted $F_e(\sigma)$. (This quantity was first introduced in \cite{NS01}.) From above, it has a useful representation:
\begin{equation}
\label{eqn: flex}
F_e(\sigma):=|J_e-C_e(\sigma)|=\inf_{A:e \in \partial A}\sum_{\{z,w\}\in \partial A} J_{zw}\sigma_z\sigma_w\ .
\end{equation}

In the same spirit as the critical values, for any edge $e$ and $\sigma \in \mathcal{G}(J_e)$, we define the set of {\it critical droplets} for $e$ in $\sigma$. 
These are the limit sets of the infimizing sequences of finite sets in the expression \eqref{eqn: c_e} of the critical value $C_e(\sigma)$.
Precisely, if $(\Lambda_n)$ is a sequence of vertex sets, we say that $\Lambda_n \to \Lambda$ if each vertex $v \in V$ is 
in only finitely many of the sets $\Lambda_n \Delta \Lambda$ (here $\Delta$ denotes the symmetric difference of sets). 
We will say that $\Lambda$ is a critical droplet for $e$ in $\sigma$ if there exists a sequence of finite vertex sets $(\Lambda_n)$ such that $\Lambda_n \to \Lambda$, 
$e \in \partial \Lambda_n$ for all $n$ and 
\[
-\sum_{\substack{\{x,y\} \in \partial \Lambda_n\\ \{x,y\} \neq e}} J_{xy} \sigma_x\sigma_y \to \sigma_eC_e(\sigma) \text{ as } n \to \infty\ .
\]
Write $CD_e(\sigma)$ for the set of critical droplets of $e$ in $\sigma$. By compactness, this set is nonempty.

Since the critical values are values of $J_e$ where there is a change in the set $\mathcal{G}(J_e)$,
it will be useful to get bounds on them that are functions of the couplings only (not of $\sigma\in \mathcal{G}(J_e)$).
In this spirit, similarly to \cite{NS01}, we define the {\it super-satisfied value} for an edge $e=\{x,y\}$ as
\begin{equation}
\label{eqn: supersat}
\mathcal{S}_e:=\min \left\{\sum_{\substack{z \neq y\\ \{x,z\}\in E}}|J_{xz}|, \sum_{\substack{z \neq x\\ \{y,z\}\in E}} |J_{yz}| \right\}\ .
\end{equation}
We will say that an edge $e$ is {\it super-satisfied} if $|J_e|> \mathcal{S}_e$.
The terminology is explained by the following fact:
by taking $A=\{x\}$ and $A=\{y\}$ in \eqref{eq:GSproperty}, one must have
\begin{equation}
\label{eq: supersat sign}
\begin{aligned}
J_e>\mathcal{S}_e &\Longrightarrow \sigma_e=+1 \text{ for all $\sigma\in \mathcal{G}(J_e)$}\\
J_e<-\mathcal{S}_e &\Longrightarrow \sigma_e=-1 \text{ for all $\sigma\in \mathcal{G}(J_e)$}\ .
\end{aligned}
\end{equation}
Moreover, for the same choice of $A$, we get from Lemma~\ref{lem: critical formula}
\begin{equation}
\label{eq: critical first bound}
C_e(\sigma)\geq-\mathcal{S}_e \text{ if $\sigma_e=+1$} \text{ and } C_e(\sigma)\leq\mathcal{S}_e \text{ if $\sigma_e=-1$.}
\end{equation}
Our next goal is to prove that in fact $|C_e(\sigma)|\leq \mathcal{S}_e$ (cf. Corollary \ref{cor: bound}). This is done by 
establishing a correspondence between the two following sets:
$$
\begin{aligned}
\mathcal{G}_{+_e}&=\{\sigma\in \Omega_2: \sigma_e=+1, \forall A\subset V \text{ finite with $e\notin\partial A$, }\sum_{\{x,y\}\in\partial A}J_{xy}\sigma_x\sigma_y\geq 0\}\ ;\\
\mathcal{G}_{-_e}&=\{\sigma\in \Omega_2: \sigma_e=-1, \forall A\subset V \text{ finite with $e\notin\partial A$, }\sum_{\{x,y\}\in\partial A}J_{xy}\sigma_x\sigma_y\geq 0\}\ .
\end{aligned}
$$
In other words, $\mathcal{G}_{\pm_e}$ are the sets of ground states on the graph $G$ minus the edge $e$, where the spins of the vertices of $e$ are restricted to have the same/opposite sign.
Note that these sets depend on the couplings but not on $J_e$.
Clearly, if $\sigma\in \mathcal{G}(J_e)$ then either $\sigma\in \mathcal{G}_{+_e}$ or $\sigma\in \mathcal{G}_{-_e}$ depending on its sign at $e$.
Moreover by \eqref{eq: supersat sign}, if $J_e>\mathcal{S}_e$, then $\mathcal{G}(J_e)\subseteq \mathcal{G}_{+e}$ and if $J_e<-\mathcal{S}_e$, then $\mathcal{G}(J_e)\subseteq \mathcal{G}_{-e}$.
Equality is derived in Corollary \ref{cor: supersat3} from the following correspondence.
\begin{prop}
\label{prop: pair}
For $\sigma\in \mathcal{G}_{+_e}$ and $\Lambda \in CD_e(\sigma)$, consider $\widetilde \sigma$  where $\sigma\Delta\widetilde \sigma=\partial \Lambda$, that is
\begin{equation}
\label{eq:mapping}
\widetilde \sigma_v = \begin{cases}
\sigma_v & v \notin \Lambda \\
-\sigma_v & v \in \Lambda
\end{cases}\ .
\end{equation}
Then $\widetilde \sigma\in \mathcal{G}_{-_e}$ and $C_e(\widetilde \sigma)\geq C_e(\sigma)$.
A similar statement holds for $\sigma \in \mathcal{G}_{-_e}$ with $\widetilde \sigma \in \mathcal{G}_{+_e}$ and $C_e(\widetilde \sigma) \leq C_e(\sigma)$.
\end{prop}
\begin{proof}
Write ${\cal D}$ for the collection of sets of edges $S$ such that $S=\partial A$ for some finite set of vertices $A$. We will use the following fact, which is verified by elementary arguments, and which was also noticed in \cite{Fink10}: if $S_1, S_2 \in {\cal D}$, then $S_1 \Delta S_2 \in {\cal D}$.

We will prove the proposition in the case $\sigma \in \mathcal{G}_{+_e}$. The other case is similar. 
Choose a sequence of finite vertex sets $(\Lambda_n)$ such that $e \in \partial \Lambda_n$ for all $n$, $\Lambda_n \to \Lambda$, and 
\[
-\sum_{\substack{\{x,y\} \in \partial \Lambda_n\\ \{x,y\} \neq e}} J_{xy} \sigma_x\sigma_y \to C_e(\sigma) \text{ as } n \to \infty\ .
\]

Write $S_n = \partial \Lambda_n$, $S = \partial \Lambda$, let $T \in {\cal D}$ and take $n$ so large that $T\cap S_n = T \cap S$ and $T \setminus S_n = T \setminus S$. Let $\widetilde J$ be the coupling configuration with value $\widetilde J_f=J_f$ for $f \neq e$ and $\widetilde J_e=C_e(\sigma)$ at $e$.
\begin{eqnarray*}
\sum_{\{x,y\} \in T} \widetilde J_{xy} \widetilde \sigma_x \widetilde \sigma_y &=& \sum_{\{x,y\} \in T \cap S_n} \widetilde J_{xy} \widetilde \sigma_x \widetilde \sigma_y + \sum_{\{x,y\} \in T \setminus S_n} \widetilde J_{xy} \widetilde \sigma_x \widetilde \sigma_y \\
&\overset{\mbox{\eqref{eq:mapping}}}{=}& -\sum_{\{x,y\} \in T \cap S_n} \widetilde J_{xy} \sigma_x\sigma_y + \sum_{\{x,y\} \in T \setminus S_n} \widetilde J_{xy} \sigma_x\sigma_y \\
&=& \sum_{\{x,y\} \in T \Delta S_n} \widetilde J_{xy} \sigma_x\sigma_y - \sum_{\{x,y\} \in S_n} \widetilde J_{xy} \sigma_x\sigma_y\ .
\end{eqnarray*}
Since $T \Delta S_n \in \mathcal{D}$ and $\sigma \in \mathcal{G}(\widetilde J)$, we have $\sum_{\{x,y\} \in T \Delta S_n} \widetilde J_{xy} \sigma_x\sigma_y \geq 0$. Therefore,
\[
\sum_{\{x,y\} \in T} \widetilde J_{xy} \widetilde \sigma_x \widetilde \sigma_y \geq -\sum_{\{x,y\} \in S_n} \widetilde J_{xy} \sigma_x\sigma_y\ .
\]
The right side tends to 0 as $n \to \infty$ by the definition of $S$ and $\widetilde{J}$, so $\sum_{\{x,y\} \in T} \widetilde J_{xy} \widetilde \sigma_x \widetilde \sigma_y \geq 0$ and  $\widetilde \sigma \in \mathcal{G}(\tilde J)$. Clearly, $\widetilde \sigma \in \mathcal{G}_{-_e}$, and by \eqref{eqn: equiv}, $C_e(\widetilde \sigma) \geq \widetilde J_e=C_e(\sigma)$.
\end{proof}
We prove three corollaries of the proposition. The first is the claimed bounds on $C_e(\sigma)$.
\begin{cor}[Super-satisfied bounds]
\label{cor: bound}
Let $e$ be an edge. If $\sigma\in \mathcal{G}(J_e)$, then
$$
|C_{e}(\sigma)|\leq \mathcal{S}_e\ .
$$
\end{cor}
\begin{proof}
We prove the bound when $\sigma\in \mathcal{G}_{+_e}$. The case $\sigma\in \mathcal{G}_{-_e}$ is similar.
The lower bound was noticed in \eqref{eq: critical first bound}.
As for the upper bound, by Lemma~\ref{prop: pair}, there exists $\widetilde \sigma\in \mathcal{G}_{-_e}$ such
that $ C_e(\sigma)\leq C_e(\widetilde\sigma)$. 
The claim then follows from $C_e(\widetilde\sigma)\leq \mathcal{S}_e$ again by \eqref{eq: critical first bound}.
\end{proof}

A useful fact about Corollary~\ref{cor: bound} is that it replaces the critical value that a priori depends on an infinite number of
couplings by a quantity that depends on finitely many. 
Another corollary is that for $J_e$ low enough or large enough, the set $\mathcal{G}(J)$ is independent of $J_e$:
\begin{cor}\label{cor: supersat3}
If $J_e> \mathcal{S}_e$, then $\mathcal{G}(J_e)=\mathcal{G}_{+_e}$. If $J_e< -\mathcal{S}_e$, then $\mathcal{G}(J_e)=\mathcal{G}_{-_e}$.
\end{cor}
\begin{proof}
Suppose first that $J_e > \mathcal{S}_e$. Then, from \eqref{eqn: equiv} and Corollary~\ref{cor: bound}, one has $\mathcal{G}(J_e)\subseteq \mathcal{G}_{+_e}$. 
Conversely, if $\sigma\in \mathcal{G}_{+_e}$, it suffices
to show that for any finite set of vertices $A$ with $e\in \partial A$
$$
J_e+ \sum_{\substack{\{z,w\}\in \partial A\\\{z,w\}\neq e}}J_{zw}\sigma_z\sigma_w\geq 0\ .
$$
By Corollary \ref{cor: bound}, we have $\mathcal{S}_e- C_e(\sigma)\geq 0$ and, using formula \eqref{eqn: c_e}, we
see that the above holds for $J_e> \mathcal{S}_e$. The proof for $\mathcal{G}_{-_e}$ is similar.
\end{proof}

Finally, we show that an infimizing sequence of sets for the critical values of an edge can never contain certain super-satisfied edges. For this we need to introduce for $e=\{x,y\}$
\begin{equation}\label{eq: vertexSSvalue}
\mathcal{S}_e^x = \sum_{\substack{\{x,z\}\in E,z \neq y}} |J_{xz}|\ .
\end{equation}
Note that by definition, $ \mathcal{S}_e=\min\{\mathcal{S}_e^x,\mathcal{S}_e^y\}$.
If $d$ and $e$ are two different edges, there exists a vertex $x$ which is an endpoint of $d$, but not of $e$. 
Having $|J_d|>\mathcal{S}_d^x$ guarantees that the edge $d$ is super-satisfied independently of the value of $J_e$.

\begin{cor}
\label{cor: supersat2}
Let $d =\{x,y\}$ and $e$ be edges such that $x$ is not an endpoint of $e$ and $|J_d|>\mathcal{S}_d^x$. If $\sigma \in \mathcal{G}(J)$ then no element $\Lambda$ of $CD_e(\sigma)$ has $d \in \partial \Lambda$.
\end{cor}
\begin{proof}
Let $\sigma\in\G(J)$ for some fixed $J$ such that $|J_d|>\mathcal{S}_d^x$.
Suppose $d \in \partial \Lambda$ for some $\Lambda \in CD_e(\sigma)$. Define $\widetilde \sigma$ as in Proposition~\ref{prop: pair}, so that $\sigma \Delta \widetilde \sigma = \partial \Lambda$. 
For $y \in \mathbb{R}$, let $J(e,y)$ be the coupling configuration that equals $J_f$ at $f \neq e$ and $y$ at $e$.
On one hand, note that, by Proposition~\ref{prop: pair}, $\sigma_d=-\widetilde \sigma_d$ and that $\widetilde \sigma\in \mathcal{G}(J(e,y))$ for either small or large values of $y$.
On the other hand, if $|J_d|>\mathcal{S}_d^x$ for $J$, then $|J_d| > \mathcal{S}_d^x$ in $J(e,y)$ for all $y\in\R$, because $x$ is not shared by $d$ and $e$.
In particular, this implies by Corollary~\ref{cor: supersat3} that the sign at the edge $d$ of the elements of $\mathcal{G}(J(e,y))$ 
must be the same for all $y\in\R$. This contradicts $\sigma_d=-\widetilde \sigma_d$.
\end{proof}

\section{The uniform measure on the set of ground states}
\label{sec: properties}

In this section we assume that 
\[
 \text{$|\mathcal{G}(J)|$ is constant $\nu$-a.s. and  $|\mathcal{G}(J)| < \infty$}\ .
\]
The first assertion holds for graphs with translation symmetry by the ergodic theorem as noted in Corollary \ref{cor: constant}.
We consider the family $(\mu_J)$ consisting of the uniform measures on $\mathcal{G}(J)$ indexed by $J\in\Omega_1$.
Recall from Proposition~\ref{prop: mu mble} that this family has a measurable dependence on $J$.
For concision, the following notation will be used throughout the paper for the product measures on $J$ and on one or two replicas of the spin configurations:
\begin{equation}
\label{eqn: M}
M=\nu(dJ) ~ \mu_J \qquad \text{ or } \qquad M=\nu(dJ) ~ \mu_J\times\mu_J\ ,
\end{equation}
where the appropriate case will be clear from the context.
In the first part, we use the monotonicity of the measure (defined below) to 
prove several facts, for example that the critical droplet
of any edge is unique. Second, we focus on the properties of the interface sampled from $M$ and
prove that, if it exists, any given edge lies in it with positive probability.

\subsection{Properties of the measure}
\label{sect: monotonicity}
We first introduce the {\it monotonicity property} of the family $(\mu_J)$. It is the analogue of the monotonicity of $\mathcal{G}(J)$ in Lemma~\ref{lem: monotone} 
at the level of measures. To define it, we give the following notation.
For any coupling configuration $J=(J_f)_{f\in E}$, fixed edge $e$ and real number $y$, let $J(e,y)$ be the coupling configuration given by
\begin{equation}
\label{eqn: J(e,s)}
(J(e,y))_f = \begin{cases}
y & \text{ if } f=e \\
J_f & \text{ if } f \neq e
\end{cases}\ .
\end{equation}
Consider any event $A\subseteq\Omega_1\times\{\sigma:\sigma_e=+1\}$.
A simple consequence of Lemma~\ref{lem: monotone}, since $|\G(J)|$ is a.s. constant, is that for almost all $J$ and for almost all $y\geq J_e$:
\begin{equation}
\label{eqn: increasing}
\begin{aligned}
\mu_J\{\sigma: (J,\sigma)\in A\}&\leq \mu_{J(e,y)}\{\sigma: (J,\sigma)\in A\}\ ;
\end{aligned}
\end{equation}
on the other hand, if $A\subseteq\Omega_1\times\{\sigma:\sigma_e=-1\}$, then for almost all $J$ and almost all $y\leq J_e$:
\begin{equation}
\label{eqn: decreasing}
\begin{aligned}
\mu_J\{\sigma: (J,\sigma)\in A\}&\geq \mu_{J(e,y)}\{\sigma: (J,\sigma)\in A\}.
\end{aligned}
\end{equation}
Similar statements hold for the product $\mu_J\times \mu_J$. 
For example, the mixed case $A\subseteq \Omega_1\times\{\sigma:\sigma_e=+1\}\times\{\sigma':\sigma'_e=-1\}$
yields for almost all $J$ and almost all $y\geq J_e$ and $y'\leq J_e$:
\begin{equation}
\label{eqn: mixed}
\begin{aligned}
\mu_J\times \mu_J \{(\sigma,\sigma'): (J,\sigma,\sigma')\in A\}&\leq \mu_{J(e,y)}\times \mu_{J(e,y')}\{(\sigma,\sigma'): (J,\sigma,\sigma')\in A\}.
\end{aligned}
\end{equation}

We refer to \eqref{eqn: increasing}, \eqref{eqn: decreasing} and \eqref{eqn: mixed} as the {\it monotonicity} of the
family $(\mu_J)$. It is a natural property to expect from a family of measures on ground states.
The results of this section, with the exception of Lemma~\ref{lem: backmodify}, are derived solely from it and no
other finer properties of the uniform measure.
The main use of the monotonicity property is to decouple the dependence on $J_e$ in $\mu_J$ from the dependence 
on $J_e$ in the considered event. This trick will appear frequently. The results of this section are stated for 
the measure $M$ in \eqref{eqn: M} with one replica of $\sigma$ for concision. They also hold for the measure $M$ on two replicas.

A useful consequence of \eqref{eqn: increasing}, \eqref{eqn: decreasing}, \eqref{eqn: mixed}, and the continuity of $\nu$ is that $\nu$-almost surely no coupling value is equal to its critical value. 
This is a special case of the next proposition, taking $B=\{e\}$ and $h_{B^c}=C_e$.
\begin{prop}
\label{prop: decoupling}
Let $B\subset E$ be a finite set of edges
and $h_{B^c}:\R^E\times \{-1,+1\}^V\to\R$ be a function that does not depend on couplings of edges in $B$.
Then for any given linear combination $\sum_{b\in B}J_bs_b$, provided that the coefficients $s_b\in\R$ are not all zero,
$$
M\{(J,\sigma): h_{B^c}(J,\sigma)=\sum_{b\in B}J_bs_b\}=0\ .
$$
The same statement holds if $h_{B^c}$ is a function of the couplings and two replicas $(J,\sigma,\sigma')\mapsto h_{B^c}(J,\sigma,\sigma')$ that does not depend on the couplings of edges in $B$.
\end{prop}
\begin{proof}
The event $\{\sigma: h_{B^c}(J,\sigma)=\sum_{b\in B}J_bs_b\}$ can be decomposed by
taking the intersection with all possible spin configurations on $B$. Suppose first that
$\sigma_b=+1$ for all $b\in B$ and define, for a given $J\in\R^E$, $J(B, y)$ for $y\in\R^B$ similarly to \eqref{eqn: J(e,s)}
$$
(J(B, y))_e=\begin{cases}
y_e, \text{ $e\in B$}\\
J_e, \text{ $e\notin B$}\ .
\end{cases}
$$
By  \eqref{eqn: increasing}, $\mu_J\{\sigma: h_{B^c}(J,\sigma)=\sum_{b\in B}J_b s_b, ~\sigma_b=+1~ \forall b\in B\}$
is smaller than the probability of the same event under the measure averaged over larger $J_b$'s. Writing $\{J_B^\geq\}$ for the event that $y_b\geq J_b$ for all $b\in B$,
\begin{eqnarray*}
&&\int \nu(dJ_B) \mu_J\{\sigma: h_{B^c}(J,\sigma)=\sum_{b\in B}J_bs_b, ~\sigma_b=+1~ \forall b\in B\} \\
&\leq&\int \nu(dJ_B) \frac{1}{\nu\{J_B^\geq\}}\int_{\{J_B^\geq\}}\nu(dy) ~ \mu_{J(B,y)}\{\sigma: h_{B^c}(J,\sigma)=\sum_{b\in B}J_b s_b, ~\sigma_b=+1~ \forall b\in B\}\ .
\end{eqnarray*}

Integrating $y$ over all of $\R^B$ and dropping $\{\sigma_b=+1~ \forall b\in B\}$ gives the upper bound:
\[
\int \nu(dJ_B) \frac{1}{\nu\{J_B^\geq\}} \int \nu(dy) \mu_{J(B,y)}\{\sigma:h_{B^c}(J(B,y),\sigma) = \sum_{b \in B} J_b s_b\} \ .
\]
Note $h_{B^c}(J(B,y),\sigma)=h_{B^c}(J,\sigma)$ as $h_{B_c}$ does not depend on couplings in $B$.
Now use Fubini:
\[
\int \nu(dy) \int d\mu_{J(B,y)}(\sigma) \left[ \int \nu(dJ_B) \nu\{J_B^\geq\}^{-1} 1_{\{\sum_{b \in B} J_b s_b = h_{B^c}(J(B,y),\sigma)\}}(J_B) \right]\ ,
\]
where $1_A(J_B)$ denotes the indicator function of the event $A$.
Because the linear combination of $J_b$'s is non-trivial and $h_{B^c}(J(B,y),\sigma)$ does not depend on $J_B$, the indicator function is equal to 1 on a set of $J_B$'s that is a hyperplane of dimension at most $|B|-1$. Therefore it is $\nu$-almost surely zero, and the inner integral equals zero. This completes the proof in the case that $\sigma_b=+1$ for all $b \in B$. To prove the other cases where $\sigma_b=-1$ for some $b \in B$, it suffices to average over $\{J_b^\leq\}$ (where this event is defined in the obvious way) for $b$ and use \eqref{eqn: decreasing}.
The proof of the second claim when $h_{B^c}$ is a function of the couplings and two replicas $(J,\sigma,\sigma')\mapsto h_{B^c}(J,\sigma,\sigma')$ is done
the same way. In the case that $\sigma_b=+1$ and $\sigma'_b=-1$, one uses \eqref{eqn: mixed} and bounds $\mu_J\times\mu_J$ by 
the average of $\mu_{J(b,y)}\times\mu_{J(b,y')}$ over $\{J_b^\geq\}\times \{J_b^\leq\}$ .
\end{proof}

One consequence of the above proposition is that the critical droplet $CD_e(\sigma)$ set cannot contain two non-flip-related elements.
In other words, infimizing sequences of finite sets of edges entering in the definition \eqref{eqn: c_e} of the critical value converge
to a unique set. This implies in particular that the mapping of Lemma~\ref{prop: pair} is well-defined.
\begin{cor}
\label{cor: droplet}
For any edge $e\in E$,  $M\{ (J,\sigma): \exists ~T_1 \neq T_2 \in CD_e(\sigma) \text{ with } T_1 \neq G \setminus T_2\} = 0$.
\end{cor}
\begin{proof}
Suppose that $CD_e(\sigma)$ contains at least two critical droplets, $T_1$ and $T_2$, not related by $T_1 = G \setminus T_2$,
with positive probability.
Let $S_1$ be the set of edges connecting $T_1$ to $T_1^c$ (similarly for $S_2$). 
Either $S_1 \setminus S_2$ or $S_2 \setminus S_1$ is non-empty. 
We may assume that $S_1 \setminus S_2$ is non-empty.
So there exists $b$ such that 
\begin{equation}\label{eq: endeq1}
M\{ (J,\sigma): \exists~ T_1, T_2 \in CD_e(\sigma) \text{ with } b \in S_1 \setminus S_2\}>0\ .
\end{equation}
Assume that $\sigma_e=+1$ and $\sigma_b=+1$; the other cases are similar. 
Define
\[
C_{b,e}(J,\sigma) = - \inf_{\substack{A:b,e \in \partial A \\ A \text{ finite}}} \sum_{\substack{\{x,y\} \in \partial A \\ \{x,y\} \neq b,e}} J_{xy} \sigma_x \sigma_y \text{ and } 
C_{e}^b(J,\sigma) = - \inf_{\substack{A:e \in \partial A \\ b \notin \partial A \\ A \text{ finite}}} \sum_{\substack{\{x,y\} \in \partial A \\ \{x,y\} \neq e}} J_{xy} \sigma_x \sigma_y\ .
\]
On the event in \eqref{eq: endeq1}, we have $C_{b,e}(J,\sigma)-J_b = C_e(J,\sigma)=C_e^b(J,\sigma)$ because $T_1$ and $T_2$ are in $CD_e(\sigma)$.
Thus \eqref{eq: endeq1} implies that
\[
M\{ (J,\sigma): \sigma_e=\sigma_b=+1, C_{b,e}(J,\sigma) - C_e^b(J,\sigma) = J_b\} > 0\ .
\]
This contradicts Proposition~\ref{prop: decoupling} using $B=\{e\}$ and $h_{B^c}(J,\sigma)= C_{b,e}(J,\sigma)-C_e^b(J,\sigma)$.
\end{proof}

We now state a lemma that will be used in Section~\ref{sec: I=0}.
By Corollary \ref{cor: supersat2}, the critical droplet cannot go through certain super-satisfied edges.
Therefore if there are such super-satisfied edges forcing the critical droplet of an edge $f$ to go through some fixed
edges $e_1$ or $e_2$, then the flexibility \eqref{eqn: flex} of $f$, by definition, cannot be smaller than both of those of $e_1$ and $e_2$.
The situation is depicted in Figure \ref{fig: magic_rung} where the super-satisfied edges appear in grey.
As in Corollary \ref{cor: supersat2}, the edges need to be super-satisfied independently of the value of $J_f$. For this reason,
we work with the value $\mathcal{S}^x_e$ defined in \eqref{eq: vertexSSvalue}. 
\begin{lem}\label{lem: cylinder}
Let  $e_1,e_2, f$ be edges.
Let $U$ be a set of edges with the property that all finite sets $A$ with $f\in\partial A$ and $\partial A \cap U = \varnothing$ must have either $e_1$ or $e_2$ in $\partial A$. 
For each $e \in U$ pick $x(e)$ to be an endpoint of $e$ that is not an endpoint of $f$.
Then
\[
M \{(J,\sigma):F_f(J,\sigma) \geq \min \{ F_{e_1}(J,\sigma),F_{e_2}(J,\sigma)\},~ \forall e\in U~ |J_e| > \mathcal{S}_e^{x(e)}\} = 1\ .
\]
\end{lem}

We will now prove two lemmas about the measure $M$ that will be useful later.
They require an extra assumption on the type of events under consideration; see for example \eqref{eqn: event monotone} and \eqref{eqn: event monotone 2}.
The results show that an event of positive probability remains of positive probability after a certain coupling modification.
They in fact provide explicit lower bounds which will be needed when dealing with weak limits of the measure $M$ in Section 4.

\begin{lem}\label{lem: SStypemod}
Let $A \subseteq \Omega_1\times \{\sigma:\sigma_e=+1\}$ be such that
\begin{equation}
\label{eqn: event monotone}
\text{If }(J,\sigma) \in A \text{ then }(J(e,s), \sigma) \in A \text{ for all }s \geq J_e.
\end{equation}
Then for each $\lambda \in \mathbb{R}$,
\begin{equation}\label{eq:monotoneestimate}
M(A,~J_e \geq \lambda) \geq (1/2) \nu([\lambda,\infty))~M(A)\ .
\end{equation}
If instead, we have $A \subseteq \Omega_1\times \{\sigma:\sigma_e=-1\}$ and
$(J(e,s),\sigma) \in A$ for all $s \leq J_e$ then 
\[
M(A,~J_e \leq \lambda) \geq (1/2) \nu((-\infty, \lambda])~M(A)\ .
\]
\end{lem}

\begin{proof}
We will prove the first statement; the second is similar. The left side of \eqref{eq:monotoneestimate} equals
\[
\int \nu(dJ_{\{e\}^c}) \left[ \int_\lambda^\infty \nu(dJ_e) \mu_J\{\sigma: (J,\sigma) \in A\}\right] \ ,
\]
where the first integral is over all couplings $J_b$ for $b \neq e$, and the second is over $J_e$. This is
\begin{eqnarray*}
&&\int \nu(dJ_{\{e\}^c}) \left[ \int_\lambda^\infty \nu(dJ_e) \frac{1}{\nu((-\infty,\lambda))} \int_{-\infty}^\lambda \mu_J\{\sigma:(J,\sigma) \in A\} \nu(dy) \right] \\
&{\overset{\mbox{\eqref{eqn: increasing}}}{\geq}}& \int \nu(dJ_{\{e\}^c}) \left[ \int_\lambda^\infty \nu(dJ_e) \frac{1}{\nu((-\infty,\lambda))} \int_{-\infty}^\lambda \mu_{J(e,y)}\{\sigma:(J,\sigma) \in A\} \nu(dy) \right] \\
&\overset{\mbox{\eqref{eqn: event monotone}}}{\geq}& \int \nu(dJ_{\{e\}^c}) \left[ \int_\lambda^\infty \nu(dJ_e) \frac{1}{\nu((-\infty,\lambda))} \int_{-\infty}^\lambda \mu_{J(e,y)}\{\sigma:(J(e,y),\sigma) \in A\} \nu(dy) \right] \\
&\geq& \nu([\lambda,\infty))~M(A,~J_e<\lambda)\ ,
\end{eqnarray*}
where the third inequality comes from dropping $ \nu((-\infty,\lambda))^{-1}$.
From this computation,
\begin{eqnarray*}
M(A,~J_e \geq \lambda) &\geq& (1/2)\left\{\nu([\lambda,\infty)) ~M(A,~J_e < \lambda) + M(A,~J_e \geq \lambda)\right\}  \\
&\geq& (1/2)\nu([\lambda,\infty))~M(A)\ .
\end{eqnarray*}
\end{proof}

The next lemma does not use the monotonicity property, but its proof is similar in spirit to the previous one.
Instead of considering coupling values that are far from the critical value, we now consider values that are close. 
To show that an event of positive probability remains of positive probability after bringing the coupling closer to the critical value, we need to use the fact that by definition, 
a ground state remains in the support of the uniform measure for all values of $J_e$ up to the critical value. 
\begin{lem}\label{lem: backmodify}
Let $c <d \in \mathbb{R}$ and $A \subseteq \{(J,\sigma):\sigma \in \mathcal{G}(J),~\sigma_e=+1\} \subseteq \Omega_1 \times \Omega_2$ be such that
\begin{equation}
\label{eqn: event monotone 2}
 \text{If $(J,\sigma) \in A$ and $J_e \geq c$ then $(J(e,y),\sigma) \in A$ for all $y\geq c$.}
\end{equation}
Then for all $d >c$,
\[
M(A,~J_e \in [c,d]) \geq \nu([c,d])~M(A,~J_e \geq c)\ .
\]
\end{lem}

\begin{proof}
 From the second condition, for a fixed $J$ with $J_e \geq c$, 
\[
\sharp \{\sigma:(J,\sigma) \in A\} \leq \sharp \{ \sigma:(J(e,y),\sigma) \in A \} \text{ for all } y \geq c\ .
\]
Since $\mu_J$ is the uniform measure and $A \subseteq \{(J,\sigma):\sigma \in \mathcal{G}(J)\}$, this implies $\nu$-almost surely
\[
\mu_J\{\sigma:(J,\sigma) \in A\} \leq \mu_{J(e,y)}\{\sigma:(J(e,y),\sigma) \in A\} \text{ for all } y \geq c\ .
\]
Therefore $M(A,~J_e \geq c)$ equals
\begin{eqnarray*}
& &\int \nu(dJ_{\{e\}^c}) \int_c^\infty \nu(dJ_{e}) \frac{1}{\nu([c,d])} \left[ \int_c^{d} \mu_J\{\sigma:(J,\sigma) \in A\} \nu(dy) \right] \\
&\overset{\mbox{\eqref{eqn: event monotone 2}}}{\leq}& \int \nu(dJ_{\{e\}^c}) \int_c^\infty \nu(dJ_{e}) \frac{1}{\nu([c,d])} \left[ \int_c^{d} \mu_{J(e,y)}\{(\sigma,\sigma'):(J(e,y),\sigma,\sigma') \in A\} \nu(dy) \right] \\
&=& \frac{\nu([c,\infty))}{\nu([c,d])} \int \nu(dJ_{\{e\}^c}) \int_c^{d} \mu_{J(e,y)}\{(\sigma,\sigma'):(J(e,y),\sigma,\sigma')\in A\}\nu(dy) \ ,
\end{eqnarray*}
which is smaller than $\frac{M(A,~J_e \in [c,d])}{\nu([c,d])}$. This implies the lemma.
\end{proof}

\subsection{Properties of the interface}
We now turn to properties of the interface $\sigma\Delta\sigma'$ under the measure
$$
M=\nu(dJ) ~ \mu_J\times \mu_J\ .
$$

The main result of this section is that if $\sigma\Delta \sigma'$ is not empty, then it can be made to contain any fixed edge of the graph with positive probability. 
A similar statement has been proved in \cite[Corollary~2.9]{ADNS10} for the metastate measure on ground states. 
The conclusion is straightforward by translation invariance in the case $G=\Z^2$.
A different approach is needed for the half-plane $G=\Z\times \N$.
For the sake of simplicity, we prove the statement in the case that the graph is planar and each face has four edges. 
The general statement for a graph $G=(V,E)$ with finite degree can be proved the same way.
\begin{prop}
\label{prop: touch}
If there exists an edge $e\in E$ such that
$M\{ (J,\sigma,\sigma'): e\in \sigma\Delta \sigma'\} >0$,
then for any edge $b\in E$,
$M\{ (J,\sigma,\sigma'):  b\in \sigma\Delta \sigma'\} >0$.
\end{prop}
Before turning to the proof, we record a fact: if $\sigma$ and $\sigma'$ are spin configurations then  a cycle (in particular, a face) of the graph cannot have an odd number of edges in $\sigma \Delta \sigma'$. This is a direct consequence of the following elementary lemma; see for example Theorem~1 in \cite{Bieche}.
\begin{lem}
\label{lem: parity}
For any finite cycle $\mathcal{C}$ in the graph $G$, the parity of $\#\{e\in \mathcal{C}: J_e<0\}$ equals the parity of $\#\{e\in \mathcal{C}: \sigma_e\neq sgn J_e\}$.
\end{lem}

The following lemma interprets the event that an edge is in the interface in terms
of the critical values of $e$ in the two ground states.
\begin{lem}
\label{lem: critical-dw}
For any edge $e$,
$M\{  (J,\sigma,\sigma'): e\in \sigma\Delta \sigma'\} >0$
if and only if
$M\{ (J,\sigma,\sigma'): C_e(J,\sigma)\neq C_e(J,\sigma')\}> 0$.
\end{lem}

\begin{proof}
$\Longrightarrow$. By assumption,
$$
M\{(J,\sigma,\sigma'): \sigma_e=+1, \sigma'_e=-1\}>0\ .
$$
By \eqref{eqn: equiv}, $\sigma\in \mathcal{G}(J)$ and $\sigma_e=+1$ together imply that
$J_e\geq C_e(J,\sigma)$. Similarly, $\sigma'\in \mathcal{G}(J)$ and $\sigma'_e=-1$ together imply that
$J_e\leq C_e(J,\sigma')$. Therefore
\begin{equation}
\label{eqn: equality}
M\{(J,\sigma,\sigma'): \sigma_e=+1, \sigma'_e=-1,C_e(J,\sigma)\leq J_e\leq C_e(J,\sigma')\}>0\ .
\end{equation}

To complete the proof, observe that Proposition~\ref{prop: decoupling} implies
$$
M\{ (J,\sigma,\sigma'): C_e(J,\sigma)= J_e \text{ or } C_e(J,\sigma') = J_e\} =0\ .
$$

$\Longleftarrow$.
We may assume that with positive probability, on the event $\{C_e(J,\sigma)=C_e(J,\sigma')\}$, $\sigma$ and $\sigma'$ have the same sign at $e$. Without loss of generality, taking $\sigma_e=\sigma'_e=+1$,
$$
M\{ (J,\sigma,\sigma'): \sigma_e=\sigma'_e=+1, C_e(J,\sigma)\neq C_e(J,\sigma')\}>0\ .
$$
In particular, there exists a deterministic $\delta>0$ such that 
$$
M\{ (J,\sigma,\sigma'): \sigma_e=\sigma'_e=+1, C_e(J,\sigma') >C_e(J,\sigma)+\delta \}>0\ .
$$
Hence there is a subset of the couplings of positive $\nu$-probability such that on this set 
$$
 \mu_J\times\mu_J\{ (\sigma,\sigma'): \sigma_e=\sigma'_e=+1, C_e(J,\sigma') >C_e(J,\sigma)+\delta \}>0
$$
Fix the couplings other than $J_e$ and take $(\sigma,\sigma')$ in the above event.
By \eqref{eqn: equiv}, we must have $J_e\geq C_e(J,\sigma)$
and $J_e\geq C_e(J,\sigma')$. 
From Proposition~\ref{prop: pair}, there exists $\sigma''\in \mathcal{G}_{-_e}$ such that $C_e(J,\sigma'')\geq C_e(J,\sigma')$.
In particular, by Corollary~\ref{cor: supersat3}, $\sigma''\in \mathcal{G}(J)$ for $J_e$ in the non-empty interval $(C_e(J,\sigma), C_e(J,\sigma''))$. Since $\mu_J$ is supported on a finite number of spin configurations, this implies that on a subset of positive $\nu$-probability
$$
 \mu_J\times\mu_J\{ (\sigma,\sigma''): \sigma_e=+1,\sigma''_e=-1 \}>0\ .
$$
Integrating over $J$ completes the proof.
\end{proof}

\begin{proof}[Proof of Proposition~\ref{prop: touch}]
By Lemma~\ref{lem: critical-dw},
it suffices to show that 
\begin{equation}
\label{eqn: to prove edge}
M\{(J,\sigma,\sigma'): C_b(J,\sigma)\neq C_b(J,\sigma')\} >0\ .
\end{equation}
Assume that
\begin{equation}
\label{eqn: touch assumption}
M\{(J,\sigma,\sigma'):\sigma_e\neq\sigma'_e\} >0\ .
\end{equation}
Without loss of generality, we can assume that $b$ and $e$ are edges of the same face. 
Otherwise, we simply apply the same argument successively on a path of neighboring faces from $b$ to $e$.
Let us denote the other edges of the square face by $\tilde{b}$ and $\tilde{e}$. 

$\sigma\Delta\sigma'$ contains $e$ with positive probability. By the paragraph preceding
the statement of the proposition, if it contains $e$ it must also contain another edge of the face.
If it contains $b$ with positive probability we are done, so suppose it contains $\tilde{e}$ with positive probability.
Suppose also that with positive probability  $\widetilde{b}$ is not in the interface.
The other case is proved the same way and is simpler. We will indicate how to deal with it at the end of the proof.

In our notation, $e,\tilde e \in \sigma \Delta \sigma'$ and $b,\tilde b \notin \sigma \Delta \sigma'$. 
Therefore $\sigma_e\neq \sigma'_e$, $\sigma_{\tilde e} \neq \sigma'_{\tilde e}$, $\sigma_b = \sigma'_b$ and $\sigma_{\tilde{b}} = \sigma'_{\tilde b}$ on this event. 
The hypothesis \eqref{eqn: touch assumption} now reduces to $M(B) >0$
for the event 
$$
B=\{(\sigma,\sigma'):  ~\sigma_{b}=\sigma'_{b},~\sigma_{\tilde b}=\sigma'_{\tilde b},  ~\sigma_{e}\neq \sigma'_{e}, \sigma_{\tilde e}\neq \sigma'_{\tilde e}\}\ .
$$
By \eqref{eqn: increasing}, for any $J$ such that $\mu_J\times\mu_J(B\cap\{\sigma:\sigma_{\tilde b}=+1\})>0$, if $J'$ is a configuration with $J'_{\tilde b}>J_{\tilde b}$, and $J'_a=J_a$ for $a \neq \tilde b$, then $\mu_{J'}\times\mu_{J'}(B)>0$. Similarly, for any $J$ such that $\mu_J\times\mu_J(B\cap\{\sigma:\sigma_{\tilde b}=-1\})>0$, if $J'$ is a configuration with $J'_{\tilde b}<J_{\tilde b}$, and $J'_a=J_a$ for $a \neq \tilde b$, then $\mu_{J'}\times\mu_{J'}(B)>0$. In particular, this implies that if $x$ is one of the endpoints of $\tilde b$ that is not also an endpoint of $b$,
$$
\int_{\{J:|J_{\tilde b}|>\mathcal{S}_{\tilde b}^x\}} \nu(dJ) ~\mu_J\times\mu_J(B) ~ >0\ .
$$
We show that
\begin{equation}
\label{eqn: to show touch}
\int_{\{J:|J_{\tilde b}|>\mathcal{S}_{\tilde b}^x\}} \nu(dJ) ~\mu_J\times\mu_J(B\cap\{(\sigma,\sigma'): C_b(J,\sigma)\neq C_b(J,\sigma')\}) ~ >0\ ,
\end{equation}
thereby proving \eqref{eqn: to prove edge} and the proposition.

The expression for the critical value $C_b(J,\sigma)$ can be written as follows. 
Let $F=\{b,\tilde{b},e,\tilde e\}$. 
For $I$ a non-empty subset of $\{\tilde b,\tilde e, e\}$, 
write $\mathcal{I}_{b,I}$ for the collection of finite sets of vertices $A$ whose boundary $\partial A$ intersected with $F$ equals the union of $\{b\}$ with $I$. This collection might be empty for some choice of $I$.
We restrict only to sets $I$ for which $\mathcal{I}_{b,I}$ is not empty.
Let
$$
C_{b,I}(J,\sigma)=\sup_{A \in \mathcal{I}_{b,I}}\left\{-\sum_{\substack{\{x,y\}\in\partial A\\ \{x,y\}\notin F }} J_{xy}\sigma_x\sigma_y\right\}\ .
$$
In this notation, the expression \eqref{eqn: c_e} becomes
$$
C_b(J,\sigma)=\max_{I\subseteq F\setminus \{b\} } \left\{\sum_{c\in I} -J_c\sigma_c + C_{b,I}(J,\sigma)\right\}\ .
$$
Let $\Lambda \in CD_b(\sigma)$, $\Lambda' \in CD_b(\sigma')$ and
note that both $\partial \Lambda$ and $\partial \Lambda'$ must contain at least one edge of the face other than $b$.
When $|J_{\tilde b}| > \mathcal{S}_{\tilde b}^x$, Corollary \ref{cor: supersat2} gives that neither can contain $\tilde{b}$, 
so they must both contain $b$ and other edges in $\{e,\tilde{e}\}$.
Therefore on this event, the above definition of the critical values reduces to
$$
C_b(J,\sigma)=\max_{I\subseteq\{e,\tilde e\}} \left\{\sum_{c\in I}-J_c\sigma_c + C_{b,I}(J,\sigma)\right\}\ .
$$
Since the max is attained, it holds on the event $\{J: |J_{\tilde b}|>\mathcal S_{\tilde b}^x\}$ that
$$
\begin{aligned}
&\mu_J\times\mu_J\Big(B\cap \{C_b(J,\sigma)=C_b(J,\sigma')\}\Big)\leq \\
&\qquad \sum_{I\subseteq\{e,\tilde e\}, I'\subseteq\{e,\tilde e\} }\mu_J\times\mu_J
\left \{\sum_{c\in I} -J_c\sigma_c + C_{b,I}(J,\sigma)=\sum_{c'\in I'} -J_{c'}\sigma'_{c'}
+ C_{b,I'}(J,\sigma')\right\}\ .
\end{aligned}
$$
The right-hand side is the same as
$$
 \sum_{I\subseteq\{e,\tilde e\}, I'\subseteq\{e,\tilde e\}  }\mu_J\times\mu_J
\Big\{ C_{b,I}(J,\sigma)- C_{b,I'}(J,\sigma')=\sum_{c\in I} J_c\sigma_c -\sum_{c'\in I'} J_{c'}\sigma'_{c'}\Big\}\ .
$$
The right-hand side of the equality in the event is a linear combination of the $J_c$'s, $c\in I\cup I'$,
where the coefficients, which we call $s_c$, can only take the values $0,\pm1,\pm2$. 
Most importantly, for each choice of $I,I'$, the $s_c$'s cannot all be zero since $I$ and $I'$ are not empty, and $\sigma_c=-\sigma'_c$ for $c\in \{e,\tilde{e}\}$. Letting $\mathcal{J}_{I,I'}$ be the set of non-zero $\{0,\pm 1, \pm 2\}$-valued vectors $s$, with each entry corresponding to an element in $I \cup I'$, we see that the above is smaller than
$$
\sum_{I\subseteq\{e,\tilde e\}, I'\subseteq\{e,\tilde e\} } \sum_{ s \in \mathcal{J}_{I,I'}} \mu_J\times\mu_J
\Big\{ C_{b,I}(J,\sigma)- C_{b,I'}(J,\sigma')=\sum_{c\in I\cup I'} J_cs_c\Big\}\ .
$$
To show \eqref{eqn: to show touch}, integrate over $\nu$ and use Proposition~\ref{prop: decoupling} 
with $B=\{e,\tilde{e}\}$ and $h_{B^c}=C_{b,I}(J,\sigma)- C_{b,I'}(J,\sigma')$.

This completes the proof in the case that $\tilde b$ is not in the interface. If the probability of this is zero (that is, if \eqref{eqn: to show touch} does not hold), then the proof is easier. We do not need to supersatisfy $J_{\tilde b}$; we simply take $I,I'$ to be subsets of $\{\tilde b, e, \tilde e\}$ and complete the proof from after equation \eqref{eqn: to show touch}.

\end{proof}

Before turning to the proof of the main result, we mention that in the case that the graph is invariant under a set of transformations (for example, translations), 
the uniform measure inherits a covariance property. Translation-covariant measures on ground states are typically not easy to construct.
The only other example known to the authors is the metastate on ground states constructed from suitable boundary conditions. An advantage
of a translation-covariant measure is that the corresponding $\nu$-averaged measure is  preserved under translations.
\begin{lem}
\label{lem: covariance}
Let $G=\Z^d$ or $G=\Z\times\N$ and suppose $|\mathcal{G}(J)|<\infty$. 
The uniform measure $\mu_J$ is translation-covariant. That is, if $T$ is a translation of $\Z^d$ or a horizontal translation of $\Z\times\N$, 
then for any $B \in \mathcal{F}_2$,
$$
\mu_{TJ}(B)=\mu_{J}\{\sigma: T \sigma\in B\} \text{ for $\nu$-almost all $J$.}
$$
In particular, the measure $M$ on $\Omega_1\times \Omega_2$ (or on $\Omega_1\times\Omega_2\times\Omega_2$) is translation-invariant.
\end{lem}
\begin{proof}
Using the fact that $|\mathcal{G}(J)|$ is constant $\nu$-almost surely, one gets
$$
\begin{aligned}
\mu_{TJ}(B)=\frac{\#\{\sigma \in \mathcal{G}(TJ): \sigma\in B\}}{|\mathcal{G}(TJ)|}=\frac{\#\{\sigma \in \mathcal{G}(J): T \sigma\in B\}}{|\mathcal{G}(J)|} 
=\mu_{J}\{\sigma: T \sigma\in B\}\ .
\end{aligned}
$$

For the second assertion, let $B'\subset \R^E\times \{-1,+1\}^V$. Define 
$T^{-1}B'=\{(J,\sigma): (TJ,T\sigma)\in B' \}$. 
Then the first claim implies that the probability of $T^{-1}B'$ is
$$
M(T^{-1}B')
= \int \nu(dJ) \mu_J \{\sigma: (TJ,T\sigma)\in B'\} = \int \nu(dJ) \mu_{TJ} \{\sigma:(TJ,\sigma)\in B'\}\ .
$$
As $\nu$ is translation-invariant, we may replace $\nu(dJ)$ by $\nu(dTJ)$ on the right side. 
The right side then equals $\int \nu(dJ) \mu_{J}(\sigma:(J,\sigma)\in B')=M(B')$ as claimed.
\end{proof}

\section{The main result on the half-plane}

\subsection{Preliminaries}
In this section, we consider the EA model on the half-plane $H=\Z\times \N$ with free boundary conditions at the bottom.
Recall from Corollary~\ref{cor: constant} that the number of ground states $|\mathcal{G}(J)|$
is non-random. We continue to assume that $|\mathcal{G}(J)| < \infty$.
Write
\[
M = \nu(dJ) \times \left( \mu_J \times \mu_J \right)\ ,
\]
where $\mu_J$ is the uniform measure on $\mathcal{G}(J)$.
We will use the notation that sampling from $M$ amounts to obtaining a triple $(J,\sigma,\sigma')$ from the space
\[
\Omega := \mathbb{R}^{E_H} \times \{-1,+1\}^{V_H} \times \{-1,+1\}^{V_H}\ ,
\]
where $E_H$ and $V_H$ denote the edges and vertices of the half-plane respectively. 
To show Theorem~\ref{thm: mainthm}, it is sufficient to prove that $M\{(J,\sigma,\sigma'): \sigma\Delta\sigma'\neq\emptyset\}=0$.
This implies that if $|\mathcal{G}(J)|<\infty$, then $|\mathcal{G}(J)|=2$. 
We will derive a contradiction from the following:
\begin{equation}
\label{eqn: assumption}
\text{ assume that } M\{(J,\sigma,\sigma'): \sigma\Delta\sigma'\neq\emptyset\}>0\ .
\end{equation}

For this purpose, a representation of the interface $\sigma\Delta\sigma'$ in the dual lattice will be 
used. Instead of thinking of an edge $e$ as being in the interface, we think of the dual edge crossing $e$ as being in it. We denote this dual edge by $e^*$. 
The interface represented this way is a collection of paths in the dual lattice.
The reader is referred to Figure~\ref{fig: dw} for an illustration of this representation.
Note that these dual paths cannot contain loops; otherwise, $\sigma$
or $\sigma'$ would violate the ground state property \eqref{eq:GSproperty}. 
Moreover, it is elementary to see that the interface cannot have dangling ends -- dual vertices with degree one in the interface (for example, using Lemma \ref{lem: parity}).
A {\it domain wall} refers to a connected component of $\sigma \Delta \sigma'$, viewed as edges in the dual lattice. 
In the case of the half-plane $G=\Z\times\N$, we call any domain wall that crosses the $x$-axis a {\it tethered domain wall}.

\begin{figure}[h]
\begin{center}
\includegraphics[height=6cm]{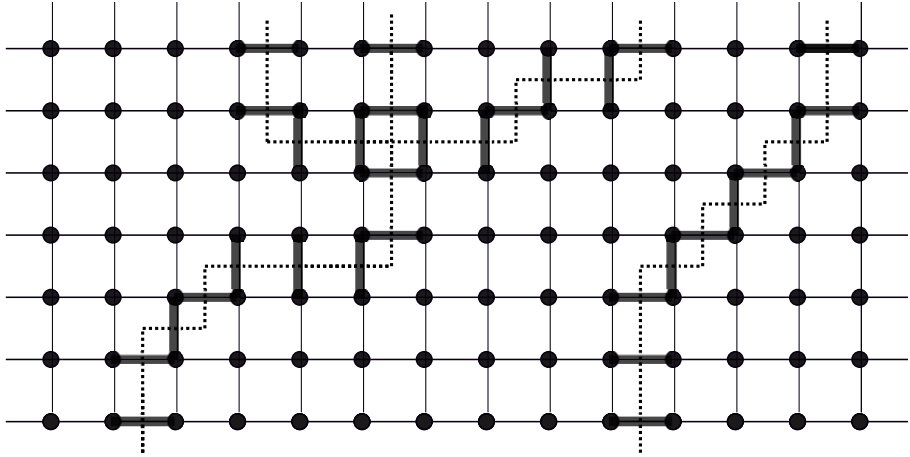}
\end{center}
\caption{An example of an interface between ground states on the half-plane. The edges in $\sigma\Delta\sigma'$ are the thick ones.
The representation of the interface as dual paths is depicted by the dotted lines. In this example, there are two domain walls and they are both tethered.}
\label{fig: dw}
\end{figure}

The method used to derive a contradiction is similar in spirit to the one in \cite{ADNS10}. 
From $M$ we construct a measure on ground states in $\Z^2$ (denoted by $\widetilde{M}$) with two contradicting properties: 
on the one hand any interface sampled from $\widetilde M$ must be disconnected; on the other hand it must be connected.
The construction of $\widetilde{M}$ is outlined below and some properties are proved. 
The proof of non-connectivity is given in Section~\ref{sec: disconnected}. 
The proof of connectivity follows the method of Newman \& Stein \cite{NS01} and is in Section~\ref{sec: connected}.

The first step is to extend the measure $M$ to include the critical values. 
This extension is needed because the critical values are not continuous functions of $(J,\sigma)$ in the product topology; they depend on the couplings in a non-local manner, as can be seen from the formula \eqref{eqn: c_e}.
Therefore their distribution is not automatically preserved under weak limits. Enlarging the probability space to include them will bypass this obstacle. For illustration, consider the event that a fixed edge $e$ has $C_e(J,\sigma) \in I$, for some fixed open interval $I$. The probability of this event is not necessarily preserved under weak limits. However, after we include the variables $C_e(J,\sigma)$ in our space, this event becomes a cylinder event and therefore its probability will behave nicely after taking limits.

We remark that a different type of extension (but with the same spirit) was done in \cite{ADNS10}. Namely, a measure called {\it the excitation metastate}
(introduced first in \cite{NS01}) was defined to include the critical values but also all information about local changes of the couplings.
Implementing this type of construction turns out to be more delicate in the case of the uniform measure. 
We therefore abandon it and turn to a simpler framework. 
The monotonicity property defined in Section \ref{sect: monotonicity} is the key tool for this approach.

For a fixed $J$, edge $e$, and $\sigma \in \mathcal{G}(J)$, recall the definition of the critical value $C_e(J,\sigma)$ from Lemma~\ref{lem: critical formula}. 
Define the map 
\begin{equation}\label{eq: phidef}
\Phi \text{ by } (J,\sigma,\sigma') \mapsto (J,\sigma,\sigma',\{C_e(J,\sigma)\}_e,\{C_e(J,\sigma')\}_e)\ ,
\end{equation}
where the last two coordinates are the collections of critical values of all edges. 
(This map is only defined for $\sigma,\sigma' \in \mathcal{G}(J)$ but this does not create a problem because the support of $\mu_J \times \mu_J$ is equal to $\mathcal{G}(J) \times \mathcal{G}(J)$.) Let $M^*$ be the push-forward of $M$ by $\Phi$ on the space 
\begin{equation}
\label{eqn: Omega*}
\Omega^* := \R^{E_H} \times \{-1,+1\}^{V_H} \times \{-1,+1\}^{V_H} \times \mathbb{R}^{E_H} \times \mathbb{R}^{E_H}\ .
\end{equation}
Sampling from $M^*$ amounts to obtaining a configuration
\[
\omega = (J,\sigma,\sigma',\{C_e\}_e,\{C_e'\}_e) \in \Omega^*\ .
\]
We have not indicated the dependence of $C_e$ on $\sigma$ and $J$, for example, because on $\Omega^*$, it is no longer a function of the other variables. 
Note that the marginal of $M^*$ on $(J,\sigma,\sigma')$ is $M$. 

We now construct a translation-invariant measure $\widetilde{M}$ on 
\[
\widetilde \Omega = \R^{E_{\Z^2}} \times \{-1,+1\}^{\Z^2} \times \{-1,+1\}^{\Z^2} \times \R^{E_{\Z^2}} \times \R^{E_{Z^2}}\ .
\]
from the measure $M^*$ using a standard procedure.
An event in $\widetilde \Omega$ that only involves, in a measurable way, a finite number of vertices of $\Z^2$ in $\sigma$ and $\sigma'$, and a finite number of edges through the couplings $J_e$ and the critical values $C_e$ and $C_e'$ will be called a {\it cylinder event}. 
Let $T$ be the translation of $\mathbb{Z}^2$ that maps the origin to the point $(0,-1)$ and for each $n \geq 0$ define
\begin{equation}
\label{eqn: M*}
M^*_n = \frac{1}{n+1} \sum_{k=0}^n T^k M^*\ .
\end{equation}
Note that the translated measure $T^k M^*$ is well-defined on cylinder events for $k$ large enough. (If it is not defined, we can take it to be zero without affecting the limit below.)
Moreover, the sequence of measures $M^*_n$ is tight. 
This is obvious for the marginal on $(J,\sigma,\sigma')$. The fact that it holds also when including the critical values is a direct consequence of Corollary~\ref{cor: bound}.
Therefore there exists a sequence $(n_k)$ such that $M^*_{n_k}$ converges as $k \to \infty$, in the sense of finite-dimensional distributions, to a translation invariant measure on $\widetilde{\Omega}$. 
Call this limiting measure $\widetilde M$. 
The weak convergence of the measures $M^*_n$ to $\widetilde M$ implies that for any event $B$ in $\widetilde{\Omega}$ 
\begin{equation}
\label{eqn: convergence sets}
\begin{aligned}
\liminf_{n\to\infty} M^*_n(B)&\geq \widetilde M(B) \text{ if $B$ is open;}\\
\limsup_{n\to\infty} M^*_n(B)&\leq \widetilde M(B) \text{ if $B$ is closed;}\\
\lim_{n \to \infty} M^*_n(B)&=\widetilde M(B) \text{ if } \widetilde M(\partial B)=0\ .
\end{aligned}
\end{equation}
(See, for example, Theorem 4.25 of \cite{Kallenberg}.)
Here we are using the fact that $\widetilde \Omega$ is metrizable, as these statements are true in general for probability measures on metric spaces.
The boundary $\partial B$ is the closure of $B$ minus its interior in $\widetilde \Omega$ (not to be confused with $\partial A$ for $A$ a finite set of vertices in the graph).
Examples of open (resp. closed) cylinder sets are $\{h(J,\{C_e\}_e,\{C_e'\}_e)\in O\}$ where $h$ is a continuous function depending only
on a finite number of edges, and $O$ is an open (resp. closed) set of $\R$. 

\begin{rem}\label{rem: rabbit}
Note that if $B$ only depends on the spins of a finite number of vertices and not on the couplings and critical values, 
actual convergence of the probability holds, since $B$ is open and closed thus $\partial B=\varnothing$. 
This same conclusion is true if $B$ is an event of the form $\{(\sigma,\sigma') \in D, J \in I\}$ for events $D$ that depend on finitely many spins and sets $I$ in some finite dimensional Euclidean space with boundary of zero Lebesgue measure. 
Indeed, it is a general fact that for any two events $B$ and $B'$, $\partial (B\cap B')\subseteq \partial B \cup \partial B'$; therefore $\partial B\subseteq \partial \{(\sigma,\sigma') \in D\}\cup \partial\{J \in I\}$. It follows that the set $\partial B$ has $\widetilde M$-measure zero, since $\widetilde M(\partial\{J \in I\})=\nu(\partial\{J \in I\})=0$ (by the continuity of $\nu$) and $\widetilde M( \partial\{(\sigma,\sigma') \in D\})=\widetilde M(\varnothing)=0$.
\end{rem}

Since $\widetilde M$ will be our object of study for the remainder of the paper, we will spend some time explaining its basic properties. Suppose $\omega=(J,\sigma,\sigma',\{C_e\}_e,\{C_e'\}_e)$ is sampled from $\widetilde M$.
First, it follows directly from the construction that $\sigma$ and $\sigma'$ are almost-surely ground states on $\mathbb{Z}^2$.
Also if we define $F_e=|J_e-C_e|$ and $F_e'=|J_e-C_e'|$ to be the {\it flexibility of the edge $e$} in $\sigma$ and in $\sigma'$, then for any finite set $A$ with $e\in \partial A$,
\begin{equation}\label{eq: tacos}
F_e\leq \sum_{\{x,y\} \in \partial A} J_{xy} \sigma_x \sigma_y\  ~\widetilde M\text{-a.s.}
\end{equation}
 and similarly for $F_e'$.
This is true because this relation holds with $M^*$-probability one on the space $\Omega^*$ and for its translates by $k$ (for $k$ large enough that $A \subseteq T^kV_H$) by \eqref{eqn: flex}.
Moreover, both sides are continuous functions of $\omega$. Thus the $\omega$'s satisfying the relation \eqref{eq: tacos} form a closed set.
Equation \eqref{eq: tacos} then follows from \eqref{eqn: convergence sets}. It remains to take the infimum over all (countably many) finite sets $A$ to conclude the following lemma.
\begin{lem}
\label{lem: flex ineq}
Let $I_e(J,\sigma):=\inf_{\substack{A:e \in \partial A \\ A \text{ finite}}}\sum_{\{x,y\} \in \partial A} J_{xy} \sigma_x \sigma_y$. For any edge $e$,
\[
\widetilde M\{F_e \leq I_e(J,\sigma)\}= 1\ .
\]
The corresponding statement holds for $\sigma'$.
\end{lem}
In other words, flexibilities produced by the weak limit procedure from half-planes are no bigger than the ones computed directly from \eqref{eqn: flex} in the full plane.
This is to be expected since the former also take into account sets $A$ that touch the boundaries of some translated half-planes. The last basic property we need is a
result analogous to Proposition~\ref{prop: decoupling} (specifically the consequence of that proposition that $M(C_e=J_e)=0$) for the weak limit $\widetilde{M}$.

\begin{lem}\label{lem: flexnonzero}
For any edge $e$,
\[
\widetilde M\{F_e = 0\} = 0\ .
\]
The corresponding statement holds for $F_e'$.
\end{lem}

\begin{proof}
It suffices to prove the statement for $F_e$. 
Because $\{F_e=0\}$ is not an open set, we cannot simply take limits in Proposition~\ref{prop: decoupling} to obtain the result.
Consider the cylinder event $\{|J_e - C_e| < \e, |J_e|<N\}$ for $\e>0$ and $N>0$.
Note that this set is open. (The cutoff in $J_e$ seems superfluous first but is useful in the estimate below.)
The conclusion will follow from \eqref{eqn: convergence sets} once we show that for each fixed $N$,
\begin{equation}\label{eq: limuniform}
T^k M\{|J_e - C_e(J,\sigma)|<\e,~|J_e|< N\}
\end{equation}
can be made arbitrarily small uniformly in $k$ (for $k$ such that $e \in T^k E_H$) by taking $\e$ small.

We prove the estimate for $k=0$ only. 
It will be clear that the same proof holds for any $k$.
Using the monotonicity \eqref{eqn: increasing} and the notation $J(e,s)$ of \eqref{eqn: J(e,s)}, we have
\begin{eqnarray*}
&&M(|J_e - C_e(J,\sigma)|<\e,~ J_e< N,~ \sigma_e=+1) \\
&=&\int \nu(dJ_{\{e\}^c}) \int_{-\infty}^{N} \nu(dJ_e) \frac{1}{\nu(J_e,\infty)} \int_{J_e}^\infty \nu(ds) ~\mu_J \{\sigma: |J_e-C_e(J,\sigma)| < \e,~ \sigma_e = +1\} \\
&\leq& \frac{1}{\nu(N,\infty)} \int \nu(dJ_{\{e\}^c}) \int \nu(dJ_e) \int \nu(ds) ~\mu_{J(e,s)} \{\sigma: |J_e-C_e(J,\sigma)| < \e,~ \sigma_e = +1\} \\
\end{eqnarray*}
We now exchange integrals using Fubini and integrate over $J_e$ first to get the upper bound
$$
\frac{1}{\nu(N,\infty)} \int \nu(dJ_{\{e\}^c}) \int \nu(ds)\int \mu_{J(e,s)}(d\sigma) \nu\{J_e: |J_e - C_e(J,\sigma)| < \e\}\ .
$$
Recall that $C_e(J,\sigma)$ does not depend on $J_e$. 
The interval $\{J_e: |J_e - C_e(J,\sigma)| < \e\}$ has length $2\e$, hence given $\delta>0$, its $\nu$-probability
can be made smaller than $\delta$, independently of  $C_e(J,\sigma)$, by the continuity of $\nu$.
We have thus shown
$$
M(|J_e - C_e(J,\sigma)|<\e,~ J_e< N,~ \sigma_e=+1)\leq \frac{\delta}{\nu(N,\infty)}
$$
Repeating the same proof, but using monotonicity in the other direction and taking $J_e > -N$,
\[
M(|J_e-C_e(J,\sigma)|< \e,~ J_e > -N,~\sigma_e=-1) \leq \frac{\delta}{\nu(N,\infty)}\ .
\]
This estimate holds for any $k$ and \eqref{eq: limuniform} can be made uniformly small by taking $\e$ small.
\end{proof}

\subsection{Non-connectivity of the interface}
\label{sec: disconnected}
In this section we show
\begin{prop}
\label{prop: not connected} 
If  \eqref{eqn: assumption} holds, then 
$
\widetilde M\{\sigma \Delta \sigma' \text{ is not connected}\} > 0\ .
$
\end{prop}

The first key ingredient is to show that with positive $M$-probability, there are infinitely many
tethered domain walls in the interface on the half-plane.
\begin{lem}
If \eqref{eqn: assumption} holds, then
with positive $M$-probability, $\sigma\Delta\sigma'$ crosses the $x$-axis. Moreover, with positive $M$-probability, $\sigma\Delta\sigma'$ has infinitely many domain walls.
\end{lem}
\begin{proof}
The first claim is a direct application of Proposition~\ref{prop: touch}. 
For the second, note that a connected component of $\sigma\Delta\sigma'$ cannot cross the $x$-axis twice.
If it did, it would contain a dual path whose union with the $x$-axis encloses a finite set of vertices $S$. We must have $\sum_{\{x,y\}\in \partial S}J_{xy}\sigma_x\sigma_y\geq 0$ and similarly in $\sigma'$ by \eqref{eq:GSproperty}.
Since $\sigma_x\sigma_y=-\sigma'_x\sigma'_y$ on $\partial S$, we conclude  $\sum_{\{x,y\}\in \partial S}J_{xy}\sigma_x\sigma_y=0$, and this has probability zero by the continuity of $\nu$. Therefore to each dual edge crossing the $x$-axis contained in $\sigma\Delta\sigma'$, there corresponds a unique connected 
component of $\sigma\Delta\sigma'$. By horizontal translation-invariance of $M$ (Lemma~\ref{lem: covariance}), if $\sigma\Delta\sigma'$ contains one such dual edge,
it must contain infinitely many. This gives the second claim.
\end{proof}

The next step is to prove that distinct connected components sampled from $M$ do not disappear 
after constructing $\widetilde M$. This is done by showing that the expected number of components intersecting a fixed box is uniformly bounded below in $k$. This is the content of the next lemma. We omit the proof; it
is exactly the same as that of \cite[Proposition~3.4]{ADNS10}.
For any $k \geq 0$ and $n \geq 1$, let 
\[
I_{n,k} = [-n,n] \times \{k\}
\]
and let $N_{n,k}$ be the number of distinct tethered domain walls that cross the line segment $I_{n,k}$.
Write $\E_M$ for the expectation with respect to $M$.
\begin{lem}
\label{lem: tethered}
For fixed $k \geq 0$, the sequence $(\E_{M} N_{n,k})_n$ is sub-additive. Therefore
\[
\lim_{n \to \infty} (1/n) \E_{ M} N_{n,k} \text{ exists }.
\]
Furthermore if \eqref{eqn: assumption} holds then there exists $c>0$ such that for all $k \geq 0$ and $n \geq 1$,
\[
\E_{M} N_{n,k} \geq cn\ .
\]
\end{lem}

\noindent
This lemma yields Proposition~\ref{prop: not connected}. We omit the proof as it is identical to \cite[Proposition~3.5]{ADNS10}. The proof there only deals with cylinder events involving only spins, and therefore limits go through using Remark~\ref{rem: rabbit}.

\subsection{The Newman-Stein technique}
\label{sec: connected}
In this section, we show 
\begin{prop}
\label{prop: connected}
$
\widetilde M(\sigma \Delta \sigma' \text{ is not connected}) =0 \ .
$
\end{prop}
This contradicts Proposition~\ref{prop: not connected} and finishes the proof of Theorem~\ref{thm: mainthm}. 
We will apply the Newman-Stein technique from \cite{NS01}. 
The idea is to construct a random variable $I$ (see below) that is defined on the event $\{\sigma \Delta \sigma' \text{ is not connected}\}$. Proposition~\ref{prop: connected} will follow from both
\begin{equation}
\label{eqn: 1}
\widetilde{M}\{I\leq 0, \sigma \Delta \sigma' \text{ is not connected}\}=0
\end{equation}
and
\begin{equation}
\label{eqn: 2}
 \widetilde{M}\{I> 0, \sigma \Delta \sigma' \text{ is not connected}\}=0\ .
\end{equation}

\subsubsection{The definition of $I$}
We first need information about the topology of interfaces $\sigma \Delta \sigma'$ sampled from $\widetilde M$. 
This is the content of the following proposition, which is analogous to Theorem~1 in \cite{NS01}.
The proof of part 1 relies on translation invariance and part 2 is a consequence of Lemma~\ref{lem: parity}. The proof of part 3 uses ideas of  Burton \& Keane \cite{BK89}.
\begin{prop}\label{prop: dwproperties}
With $\widetilde M$ probability one, the following statements hold.
\begin{enumerate}
\item If $\sigma \Delta \sigma'$ is nonempty, then it has positive density. 
\item If $\sigma \Delta \sigma'$ is nonempty, then it does not contain any dangling ends or three-branching points.
\item  If $\sigma \Delta \sigma'$ is nonempty, then it contains no four-branching points. In particular, each dual vertex in the domain wall has degree two; thus each domain wall is a doubly infinite dual path.
Moreover, each component of the complement (in $\mathbb{R}^2$) of $\sigma \Delta \sigma'$ is unbounded and has no more than two topological ends in the following sense. If $C$ is such a component then for all bounded subsets $B$ of $\mathbb{R}^2$, the set $C \setminus B$ does not have more than two unbounded components.
\end{enumerate}
\end{prop}
Parts 2 and 3 of the proposition tell us that the regions between domain walls are topologically either strips or half-spaces. 
This implies that there is a natural ordering on domain walls: each domain wall has 0, 1 or 2 well-defined neighboring domain walls.
In particular, dual paths from one domain wall to a neighboring one are well-defined:
\begin{df}\label{def: rung}
A {\bf rung} is a non-self intersecting finite dual path that starts at a dual vertex in a domain wall and ends at a dual vertex in a different domain wall. No other dual vertices on the path are in a domain wall.
\end{df}
Let $h=h(\omega)$ be the first horizontal edge in the interface starting from the origin to the right.
For almost every configuration $\omega$ such that the interface is nonempty, such an $h$ exists because of translation and rotation invariance of $\widetilde M$. So we can define 
\[
I = \inf_R E(R)\ ,
\]
where the infimum is over all rungs $R$ touching the domain wall of $h^*$ and $E(R)$ is the energy:
\[
E(R) = \sum_{\{ x,y \}^* \in R} J_{xy} \sigma_x \sigma_y\ .
\]
See Figure~\ref{fig: rung} for a depiction of $h$ and a rung under consideration.

\begin{figure}[h]
\begin{center}
\includegraphics[height=8cm]{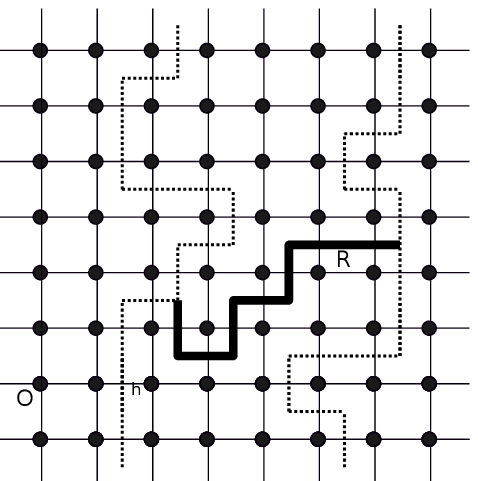}
\caption{An example of rung from the domain wall of $h$ to another domain wall.}
\label{fig: rung}
\end{center}
\end{figure}
Note that since no edge of a rung is in the interface $\sigma \Delta \sigma'$, we must have $\sigma_x\sigma_y = \sigma_x'\sigma_y'$ for all edges $\{ x,y \} \in R$. 
Therefore in the definition of $E(R)$ it does not matter if we choose $\sigma$ or $\sigma'$ to perform the computation.

\subsubsection{$I\leq 0$ has zero probability}

We will now
\begin{equation}\label{eq: assumption3}
\text{assume that }\widetilde{M}\{I\leq 0, \sigma \Delta \sigma' \text{ is not connected}\}>0\ ,
\end{equation}
and derive a contradiction.
For a dual edge $e^*$, $\e>0$ and a positive integer $K$, let $A_e(\e,K)$ be the event that (a) $e^*$ is in a rung (between any two domain walls) with energy less than $\e$ and (b) this rung has length (number of dual edges) at most $K$. Whenever $I\leq0$ and $\e>0$ there must exist a rung starting from the domain wall containing $h^*$ with energy less than $\e$. So $A_e(\e,K)$ occurs for some $e$ and $K$, and under \eqref{eq: assumption3}, there exists $\e>0$ and $K$ such that 
\[
\sum_{e \in \mathbb{E}^2} \widetilde M(A_e(\e,K)) \geq  \widetilde M \left(\cup_{e \in \mathbb{E}^2} A_e(\e,K)\right) > 0\ .
\]
By translation invariance, $\widetilde M(A_e(\e,K))>0$ for all $e$.

Let us say that dual edges $e_1^*$ and $e_2^*$ are {\it on the same side} of a domain wall $D$ if they both have a dual endpoint in the same connected component of the complement of $D$. The following lemma is the same as Lemma~1 in \cite{NS01}.
\begin{lem}\label{lem: infinitelymanyedges}
With $\widetilde M$-probability one, the following holds. If $\sigma \Delta \sigma'$ is not connected, then for each domain wall $D$, either there are infinitely many dual edges $e^*$ touching $D$ such that $A_e(\e,K)$ occurs (in {\it both} directions along $D$ and on each side of $D$) or there are zero.
\end{lem}

\begin{proof}
For an edge $e$, let $B_e(\e,K)$ be the event that (a) $A_e(\e,K)$ occurs and (b) there exists a domain wall $D$ such that $e^*$ touches $D$ and in at least one direction on $D$, there are no endpoints of dual edges $h^*$ for which $A_h(\e,K)$ occurs for the same domain wall $D$ on the same side. For each $e$ such that $B_e(\e,K)$ occurs we may associate $e$ to a domain wall $D$. Note that in each realization $\omega$ in the support of $\widetilde M$, there are at most 4 edges associated with each domain wall (counting two directions and two sides of the domain wall). 

Let $B(n)$ be the box of side length $n$ centered at the origin, and let $N_n$ be the number of domain walls which have a dual vertex in $B(n)$. Last, let us use the notation that $e \in B(n)$ if both of $e$'s endpoints are in $B(n)$. The above arguments imply that
\begin{eqnarray*}
\sum_{e \in B(n)} \widetilde M (B_e(\e,K)) &=& \E_{\widetilde M} \left( \sum_{e \in B(n)} 1(B_e(\e,K)) \right)
\leq 4 \E_{\widetilde M} N_n\ .
\end{eqnarray*}
Here $\mathbb{E}_{\widetilde M}$ stands for expectation with respect to $\widetilde M$. Distinct domain walls  do not intersect so we can associate to each dual edge of the outer edge boundary $\partial_e B(n)$ (that is, having one endpoint in $B(n)$ and one in $B(n)^c$) at most one domain wall that contains it. Therefore for some suitable constants $C_1,C_2>0$
\[
\frac{1}{|B(n)|} \sum_{e \in B(n)} \widetilde M(B_e(\e,K)) \leq C_1 \frac{1}{|B(n)|} |\partial_e B(n)| \leq C_2 |B(n)|^{-1/2} \to 0
\]
as $n\to\infty$. By translation invariance, $\widetilde M(B_e(\e,K))$ is the same for all $e$ and thus equals $0$, completing the proof.

\end{proof}

\begin{rem}\label{rem: rem1}
Although the previous lemma was stated for the events $A_e(\e,K)$, the same proof can be used for a number of different events like $A_e(\e,K)$. In \cite{NS01}, these events were called ``geometrically defined.'' Examples of such events are (a) the event that $e^*$ is in a domain wall and is adjacent to a rung with a specified energy and (b) the event that $e^*$ is in a domain wall and has a specified flexibility in $\sigma$ or $\sigma'$. We will use these facts later in Section~\ref{sec: I=0}. Note that it is not enough to use only translation-invariance in the proof, as we would need to use (random) translations along a domain wall.
\end{rem}

\begin{proof}[Proof of \eqref{eqn: 1}]
For an edge $e$, $\e>0$ and a positive integer $K$, let $A_e^0(\e,K)$ be the event that $A_e(\e,K)$ occurs and one of the endpoints of $e^*$ is in the domain wall of $h^*$. If $I \leq 0$, then for each $\e$ there exists $K$ such that $A_e^0(\e,K)$ occurs. By Lemma~\ref{lem: infinitelymanyedges}, we may find infinitely many dual edges $e_n^*$ and $f_n^*$ (in both directions along the domain wall of $h^*$ but on the same side as $e^*$) such that $A_{e_n}(\e,K)$ and $A_{f_n}(\e,K)$ occur. The $e_n$'s are chosen in one direction and the $f_n$'s in the other. Let $R_n$ be a rung corresponding to $e_n$ and let $S_n$ corresponding to $f_n$. By relabeling the sequences $(e_n)$ and $(f_n)$ we may ensure that $R_n$ does not intersect $S_n$ for any $n$. (Here we are using the fact that the rungs have length at most $K$ and so  for a fixed $n_0$, there are finitely many $n$'s such that $R_{n_0}$ intersects $S_n$.) Calling $D_0$ the domain wall containing $h^*$, both rungs $S_n$ and $R_n$ connect $D_0$ to the same domain wall, say, $D_1$.

Since $R_n$ and $S_n$ are disjoint, the dual path $P$ consisting of $R_n$, $S_n$, the piece of $D_0$ between $e_n^*$ and $f_n^*$ (call it $P_0$) and the corresponding piece of $D_1$ between the intersection points of $R_n$ and $S_n$ with $D_1$ (call it $P_1$) is a circuit in the dual lattice. See Figure \ref{fig: circuit} for a depiction.
\begin{figure}[h]
\label{fig: circuit}
\begin{center}
\includegraphics[height=8cm]{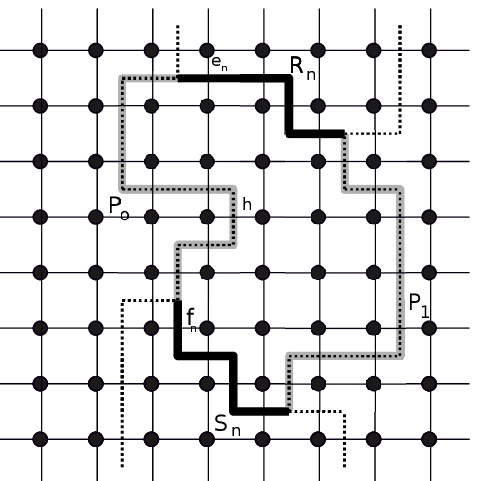}
\caption{The rungs $R_n$ and $S_n$ form a circuit in the dual lattice together with the shaded dual paths $P_0$ and $P_1$ of the domain walls.}
\end{center}
\end{figure}
The spin configurations $\sigma$ and $\sigma'$ sampled from $\widetilde M$ are ground states, hence
\[
\sum_{\{x,y\}^* \in P} J_{xy} \sigma_x \sigma_y \geq 0 \text{ and } \sum_{\{x,y\}^* \in P} J_{xy} \sigma'_x \sigma'_y \geq 0\ .
\]
For each edge $\{x,y\}$ whose dual edge is in either $R_n$ or $S_n$, we have $\sigma_x \sigma_y = \sigma'_x \sigma'_y$. For each edge $\{x,y\}$ whose dual edge is on either $P_1$ or $P_0$ we have $\sigma_x \sigma_y = - \sigma'_x \sigma'_y$. Using the fact that the energies of the rungs $R_n$ and $S_n$ are below $\e$, the above two inequalities reduce to
\[
\left| \sum_{\{x,y\}^* \in P_0} J_{xy}\sigma_x \sigma_y + \sum_{\{x,y\}^* \in P_1} J_{xy} \sigma_x \sigma_y \right| < 2 \e\ ,
\]
and so $\sum_{\{x,y\}^* \in P} J_{xy} \sigma_x \sigma_y < 4\e\ .$ As $\e$ is arbitrary, the edge $h$ has flexibility zero by Lemma~\ref{lem: flex ineq} (since $I_h=0$). By Lemma~\ref{lem: flexnonzero}, this has zero probability, proving \eqref{eqn: 1}.
\end{proof}

\subsubsection{$I>0$ has zero probability}
\label{sec: I=0}

We now show \eqref{eqn: 2} by assuming
\begin{equation}
\label{eqn: assum I>0}
 \widetilde{M}\{I> 0, \sigma \Delta \sigma' \text{ is not connected}\}>0\ .
\end{equation}
and deriving a contradiction.
The idea is that if $I > 0$ then we can find one rung near the origin whose energy we can lower by making a local modification to the couplings. 
The contradiction follows because the first edge in this rung will be the only one touching its domain wall with a certain energy property. This violates a variation of Lemma~\ref{lem: infinitelymanyedges}. 
In this section we will write $I=I(\omega)$ to emphasize the dependence of $I$ on the configuration $\omega \in \widetilde \Omega$.

For each edge $e$ let
\[
\widetilde F_e:= \min\{ F_e, F'_e \}\ .
\]
By Lemma~\ref{lem: flexnonzero}, 
\begin{equation}\label{eq: doubleflexnonzero}
\widetilde M(\widetilde F_e > 0 \text{ for all } e) = 1\ .
\end{equation}

The rest of this subsection will serve to prove the following proposition. 
Fix $\e>0$ and let $f$ be the edge connecting $(1,0)$ and $(1,1)$. 
Also define $g$ to be the edge connecting the origin to $(1,0)$. Let $X_\e$ be the intersection of the following events:
\begin{enumerate}
\item $\sigma\Delta\sigma'$ is disconnected and $I>0$;
\item $g^* \in \sigma \Delta \sigma'$;
\item $f^*$ is in a rung $R$ that satisfies $E(R) < I(\omega)+\e/2$.
\end{enumerate}
Note that on $X_\e$, the edge $h$ (used in the definition of $I$ in the previous section) equals $g$.

\begin{prop}\label{prop: xe}
If \eqref{eqn: assum I>0} holds, there exists $\e_0$ such that for all but countably many $0<\e<\e_0$, $\widetilde M\left(X_\e,~ \widetilde F_f>\e\right)>0$.
\end{prop}

\begin{proof}
We begin by finding deterministic replacements for many local quantities. 
Let $E_1$ be the event that $g^*\in\sigma\Delta\sigma'$, $f^*\notin\sigma\Delta\sigma'$, $\sigma\Delta\sigma'$ is disconnected and $I>0$. By translation invariance and by the assumption \eqref{eqn: assum I>0}, we
have $\widetilde{M}(E_1)>0$. We denote the domain wall of $g^*$ by $D_0(\omega)$ for $\omega\in E_1$.
By \eqref{eq: doubleflexnonzero}, we may choose $\e_0>0$ such that whenever $0<\e<\e_0$,
\[
\widetilde M(E_1,~\text{there exists } e^* \in D_0(\omega) \text{ such that } \widetilde F_e>\e) > 0\ .
\]
Furthermore, note that the distribution of $\widetilde F_e$ (for any edge $e$) under the measure $\widetilde M$ can only have countably many atoms. We fix any such $0<\e< \e_0$ in the complement of this set for the rest of the proof, so that
\begin{equation}\label{eq: rigging2}
\widetilde M(\widetilde F_e = \e \text{ for some } e) = 0\ .
\end{equation}
Let $E_2=E_1\cap\{\exists e^* \in D_0(\omega) \text{ such that } \widetilde F_e>\e\}$. 

If $\omega\in E_2$, we may find a rung $R(\omega)$ touching $D_0(\omega)$ such that 
\begin{equation}\label{eq: Rdef}
E(R(\omega))<I(\omega)+\e/2\ .
\end{equation}
This is by the definition of $I(\omega)$.
Let $f^*(\omega)$ be the dual edge in $R(\omega)$ that touches $D_0(\omega)$. There are countably many choices for $f^*(\omega)$, so we may find a deterministic $\widetilde f^*$ such that
\[
\widetilde M(E_2,~f^*(\omega)=\widetilde f^*)>0\ .
\]
In fact, by rotation and translation invariance we can take $\widetilde f^*$ to be the fixed dual edge $f^*$:
\[
\widetilde M(E_2,~f^*(\omega)=f^*)>0\ .
\]

By an argument identical to that given in Lemma~\ref{lem: infinitelymanyedges}, for  $\widetilde M$-almost all $\omega \in E_2$, there are infinitely many dual edges $e^* \in D_0(\omega)$ (in both directions along $D_0(\omega)$) for which $\widetilde F_e>\e$. (See Remark~\ref{rem: rem1}.)
Therefore, for $\widetilde M$-almost every $\omega\in E_2$, we may find dual edges $e_1^*(\omega)$ and $e_2^*(\omega)$ on $D_0(\omega)$ such that the piece of $D_0(\omega)$ from $e_1^*(\omega)$ to $e_2^*(\omega)$ contains $g^*$ and such that $\widetilde F_{e_1}$ and $\widetilde F_{e_2}$ are bigger than $\e$. 
For any $N$, let $B(0;N)$ be the box of side length $N$ centered at the origin and for a spin configuration $\sigma$, let $\sigma_N$ be the restriction to $B(0;N)$. There are only countably many choices, so we may find deterministic values of $e_1,e_2,N,\sigma_N,\sigma_N'$ and $R$ (whose first dual edge is $f^*$) such that with positive $\widetilde M$-probability on $E_2$:
\begin{enumerate}
\item$ B(0;N/2)$ contains $R$, $e_1^*$, $e_2^*$ and the piece of $D_0(\omega)$ between $e_1^*$ and $e_2^*$;
\item $\sigma(\omega)\Big|_{B(0;N)}=\sigma_N$, $\sigma'(\omega)\Big|_{B(0;N)}=\sigma'_N$, $\widetilde F_{e_1} > \e, \widetilde F_{e_2}>\e$ ;
\item $R$ is a rung with $E(R) < I(\omega) + \e/2$.
\end{enumerate}
Call $E_3$ the set of configurations satisfying the three above conditions. By construction, $\widetilde M(E_3\cap \{I>0\})>0$.
Note that by the choice of $\sigma_N$ and $\sigma_N'$, their interface contains $g^*$, $e_1^*$ and $e_2^*$ 
(and they are all connected through a single domain wall in $B(0;N)$), but the interface does not contain $f^*$. In addition, if $E_3$ occurs then $R$ is a rung, and $\sigma \Delta \sigma'$ must be disconnected. Therefore $E_2$ contains $E_3 \cap \{I>0\}$. The same arguments also show that 
\begin{equation}\label{eq: E3eq}
X_\varepsilon \supseteq E_3 \cap \{I>0\}\ .
\end{equation}

Now, write $D$ for the (deterministic) set of edges in $B(0;N)$ that are in $\sigma_N \Delta \sigma'_N$ and can be connected to $g^*$ by a path of dual edges in $\sigma_N \Delta \sigma'_N$ that stay in $B(0;N)$. 
This is just the connected ``piece'' of $D_0(\omega)$ in $B(0;N)$ for configurations $\omega \in E_3$. 
Let $f_1^*, \ldots, f_n^*$ be the dual edges with both endpoints in $B(0;N)$ that are (a) incident to $D$, (b) not in $\sigma_N \Delta \sigma_N'$ and (c) not equal to $f$. 
A depiction of these definitions is given in Figure \ref{fig: magic_rung}.

\begin{figure}[h]
\begin{center}
\includegraphics[height=8cm]{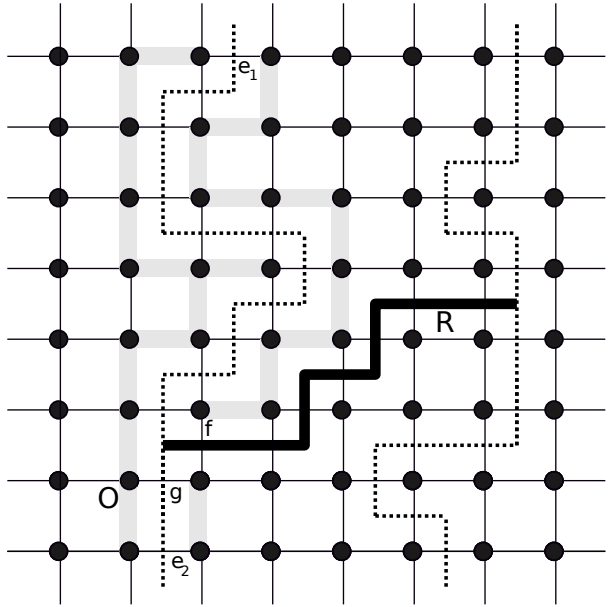}
\caption{Depiction of definitions on the event $E_3$. The two domain walls are the dual dotted lines. The rung $R$ is the thick path between the two domain walls.
The edges $f_1, f_2,\dots,f_n$ along the domain wall are the grey edges.}
\label{fig: magic_rung}
\end{center}
\end{figure}

We claim that we can order the $f_i^*$'s so that for each $i=1, \ldots, n-1$, $f_i$ has an endpoint $x_i$ that does not touch any edge from the set $\{f_{i+1}, \ldots, f_n, f\}$ (note here we are considering edges, not dual edges). To explain why this is true, we consider the graph whose edge set is equal to the union of the $f_i$'s (in the original lattice). Note that if $C_1, \ldots , C_p$ are the components of this graph then it suffices to give an ordering of each component and then concatenate these orderings together. So we may consider just one component, say, $C_1$. We will choose the edges $g_1, \ldots, g_k$ of $C_1$ in reverse order, so that our final ordering of $C_1$ will be $g_k, \ldots, g_1$. The desired condition on the $f_i$'s becomes the following for the $g_i$'s: for each $i=1, \ldots, k$, $g_i$ has an endpoint that does not touch any edge from the set $\{f, g_1, \ldots, g_{i-1}\}$. 
 
We now note that the graph whose edges are $f,f_1, \ldots, f_n$ does not contain any cycles. 
If there were a cycle then it would force the interface $\sigma \Delta \sigma'$ in the dual graph to have one too, which is impossible. Therefore the component $C_1$ above can have at most one edge that touches $f$. If there is such an edge, we let $g_1$ be it; otherwise, we choose $g_1$ arbitrarily in $C_1$. We now add edges in steps: at each step $j \geq 2$ we let $G_j$ be the current connected subgraph of $C_1$ (that is, the graph whose edges are $\{g_1, \ldots, g_{j-1}\}$) and add $g_j$ to our collection of edges so that it connects $G_j$ to its complement. This is always possible because $C_1$ does not contain a cycle. We finish at step $k$ with the desired ordering of the $g_j$'s, which, when reversed, gives the desired ordering of $C_1$.

We claim
\begin{equation}\label{eq: SSNS}
\widetilde M(E_3,~\cap_{i=1}^n\{|J_{f_i}| > \mathcal{S}_{f_i}^{x_i}\})>0\ ,
\end{equation}
where $\mathcal{S}_{f_i}^{x_i}$ is the super-satisfied value of the edge $f_i$ defined in \eqref{eq: vertexSSvalue}.
Essentially, the claim means that the event $E_3$ is somewhat stable under modifications of couplings.
Equation~\eqref{eq: SSNS} will be proved in the lemma below. 
We first show how this implies the claim of the proposition using Lemma~\ref{lem: cylinder}. 
Let $U$ be the set of all $f_i$'s.
Note that by construction, for any finite set $A$ such that $f\in\partial A$ and $\partial A\cap U=\varnothing$, we must have $e_1$ or $e_2$ in $\partial A$.
Let $\widetilde G$ be the event that $F_f \geq \min\{F_{e_1},F_{e_2}\}$ and $|J_{f_i}|\geq\mathcal{S}_{f_i}^{x_i}$ for all $f_i \in U$. 
The probability of $\widetilde G$ under any translates of $M$ is equal to that of $\widetilde G$, which is $1$ by Lemma~\ref{lem: cylinder}. 
On the other hand, $\widetilde G$ is a closed event so $\widetilde M(\widetilde G)$ is no smaller than $\limsup_k M_k^*(\widetilde G) = 1$.
This implies from \eqref{eq: SSNS}
\[
 \widetilde  M(E_3,~ \widetilde F_{f}>\e) \geq \widetilde M(E_3, \widetilde F_f >\e, ~\cap_{i=1}^n\{|J_{f_i}| \geq \mathcal{S}_{f_i}^{x_i}\} ) = \widetilde  M(E_3,~\cap_{i=1}^n\{|J_{f_i}| \geq \mathcal{S}_{f_i}^{x_i}\})>0\ .
\]
Since $X_\e\supseteq E_3 \cap \{I>0\}$ (and $\widetilde M(I=0)=0$ by \eqref{eqn: 1}), this concludes the proof of Proposition~\ref{prop: xe}.
\end{proof}

\begin{lem}\label{lem: almostSS}
Let $E_3$ be the event defined above \eqref{eq: E3eq}. Define $f_i$ and $x_i$, $i=1,...,n$ as above \eqref{eq: SSNS}.
If $\widetilde M(E_3)>0$, then $\widetilde M(E_3,~\cap_{i=1}^n\{|J_{f_i}| > \mathcal{S}_{f_i}^{x_i}\})>0$.

\end{lem}
\begin{proof}
Write $S$ for the set of dual edges in $B(0;N)$ that are not equal to any of the $f_i^*$'s or to $f^*$.
Since $\widetilde M(E_3)>0$, we can choose $\lambda>0$ such that
\[
\widetilde M(E_3,~|J_e|\leq \lambda \text{ for all } e^* \in S)>0\ .
\]
Write $E_4$ for this event. We will show that $\widetilde M(E_4,~\cap_{i=1}^n\{|J_{f_i}| > \mathcal{S}_{f_i}^{x_i}\})>0$.
This will follow if we find positive numbers $a_1, \ldots, a_n,$ $b_1, \ldots, b_n$ such that the following hold:
\begin{enumerate}
\item $a_i<b_i$ for all $i$;
\item $a_{i+1}> 4b_i$ for $i=0,1, \ldots, n-1$ and $b_0:= \lambda$.
\item $\widetilde M(E_4,~|J_{f_i}|\in[a_i,b_i] \text{ for all $i$}\})>0$;
\end{enumerate}
These conditions imply that $|J_{f_i}|>\mathcal{S}_{f_i}^{x_i}$ for all $i$, as $|J_{f_i}|\geq a_i> 4 b_{i-1}$ and $4b_{i-1}>\mathcal{S}_{f_i}^{x_i}$. Here we are using the fact that $x_i$ does not touch the set $\{f,f_{i+1}, \ldots, f_n\}$.
For $q>1$, define
\[
E_4^q:= E_4 \cap \{|J_i| \in [a_i,b_i] \text{ for all }i=1, \ldots, q-1\}
\] 
and for $q=1$ define $E_4^q := E_4$. 
We will proceed by induction to show that if $\widetilde M(E_4^{q})>0$ then $\widetilde M(E_4^{q+1})>0$ with appropriately chosen $a_q,b_q$, for $q=1,...,n-1$.
Note that $\widetilde M(E_4)>0$. The case $q=n-1$ gives the desired conclusion.
For the rest of the proof, we assume that the spins at the endpoints of $f_q$ are the same. The subsequent argument is similar in the other case.
The idea is to use Lemma~\ref{lem: SStypemod}, which shows that the probability mass is somewhat conserved when one value of the coupling is increased for 
events satisfying \eqref{eqn: event monotone}. Two obstacles have to be overcome. 
First, the properties of $M$ (in particular, the monotonicity property) needed in Lemma~\ref{lem: SStypemod} do not directly carry over under weak limits to $\widetilde M$.
Therefore, we need to go back to $M$ to apply the lemma. Second, weak convergence of the measures applies to cylinder events. Note that, from its definition, $E_4^q$ is an intersection
of a finite number of cylinder events except for $\{R \text{ is a rung with } E(R) < I(\omega) + \e/2\}$. 
To apply Lemma~\ref{lem: SStypemod}, we thus need to find a cylinder approximation for this condition.

Let $\widetilde B^R \subseteq \widetilde \Omega$ be the event $\{R \text{ is a rung with } E(R) < I(\omega) + \e/2\}$  intersected with the event $\{g^* \in \sigma \Delta \sigma'\}$. 
We will first define a double sequence of cylinder events $(\widetilde B^R_{j,l})$ in $\widetilde \Omega$ with 
\begin{equation}\label{eq: symmetricdiff}
\lim_{j \to \infty} \limsup_{l \to \infty} \widetilde M(\widetilde B^R \Delta \widetilde B^R_{j,l}) = 0\ ,
\end{equation}
where $\Delta$ represents the symmetric difference of events. 

Let $B(0;j)$ be the box of side-length $j$ centered at $0$ and let $l \geq j$. 
For arbitrary spin configurations $\sigma$ and $\sigma'$, the interface $\sigma \Delta \sigma'$ splits into different connected components in the following way. 
Two dual edges in $B(0;j) \cap \sigma \Delta \sigma'$ (that is, they have both endpoints in $B(0;j)$) are said to be {\it $l$-connected} if they are connected by a path of dual edges in $\sigma \Delta \sigma'$, all of which remain in $B(0;l)$. Let $D_0(j,l), D_1(j,l), \ldots, D_t(j,l)$ be the $l$-connected components of such edges in $B(0;j)$, where $D_0(j,l)$ is the connected component containing $g^*$ (if one exists). 
Call these the {\it $(j,l)$-domain walls} (see Figure \ref{fig: (j,l)}).
We define a {\it $(j,l)$-rung} as a finite path of dual edges in $B(0;j)$ which starts in a $(j,l)$-domain wall and ends in a different one, and no dual vertices on the path except for the starting and ending points are on a $(j,l)$-domain wall. 
\begin{figure}[h]
\begin{center}
\includegraphics[height=7cm]{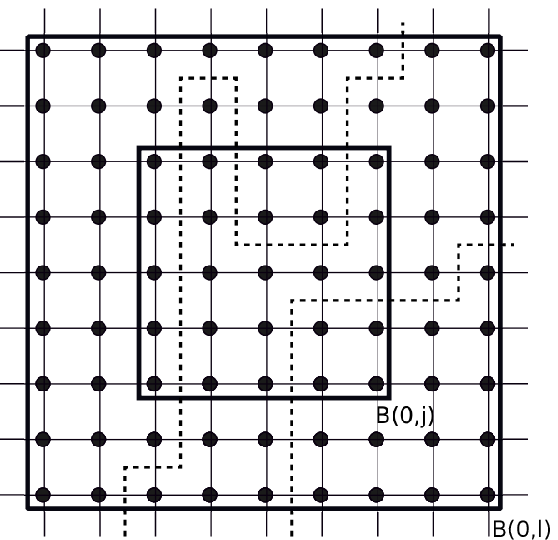}
\caption{In this figure, when we restrict the interface to $B(0;j)$, there are three components (connected inside this box). However, two of them are $l$-connected. Therefore, there are two $(j,l)$-domain walls in $B(0;j)$. }
\label{fig: (j,l)}
\end{center}
\end{figure}

On the event $\widetilde B^R$, $R$ is a $(j,l)$-rung for all $l \geq j\geq N$. Let $\widetilde B^R_{j,l}\subseteq \widetilde \Omega$ be the event that 
\begin{enumerate}
\item $g^* \in \sigma \Delta \sigma'$;
\item $R$ is a $(j,l)$-rung;
\item no other $(j,l)$-rung between $D_0(j,l)$ and another $(j,l)$-domain wall has energy less than the energy of $R$ minus $\e/2$. 
\end{enumerate}

We start by showing that
\begin{equation}\label{eq: stapler1}
\lim_{j \to \infty} \limsup_{l \to \infty} \widetilde M(\widetilde B^R \setminus \widetilde B^R_{j,l}) = 0\ .
\end{equation}
Consider $\omega \in \widetilde B^R$. 
It suffices to prove that there exists $J(\omega)$ and for each $j\geq J(\omega)$ there is a $L(j,\omega)$ such that 
\begin{equation}\label{eq: decondition}
j\geq J(\omega) \text{ and } l \geq L(j,\omega) \text{ implies } \omega \in \widetilde B^R_{j,l}\ .
\end{equation}
This implies \eqref{eq: stapler1} because if $\widetilde B^R \setminus \widetilde B^R_{j,l}$ occurs then either $j \leq J(\omega)$ or both $j \geq J(\omega)$ and $l \leq L(j,\omega)$. 
Therefore the limit in \eqref{eq: stapler1} is bounded above by
\[
\lim_{j \to \infty} \lim_{l \to \infty} \left[ \widetilde M(j \leq J(\omega)) + \widetilde M(l \leq L(j, \omega)) \right] = 0\ .
\]
Take $j\geq N$. 
Note that there are at most $|B(0;j)|$ number of $(j,l)$-domain walls in $B(0;j)$. 
We claim that there exists $L(j)$ such that all $(j,l)$-rungs are rungs for $l\geq L(j)$. Indeed, if $S$ is a rung then it is plainly a $(j,l)$-rung. 
On the other hand, if $S$ is a $(j,l)$-rung, then either it connects distinct domain walls in $\sigma\Delta\sigma'$ or simply two pieces of  the same domain wall of
$\sigma\Delta\sigma'$ that are $l$-connected for $l$ large enough. Now, since $\omega \in \widetilde B^R$, we must have $g^* \in \sigma \Delta \sigma'$. 
Moreover, $R$ is a rung and so it is also a $(j,l)$-rung for any $l$. 
By definition of $\widetilde B^R$, no rung touching $D_0(\omega)$ can have energy less than the energy of $R$ minus $\e/2$. 
Therefore for $l \geq L(j)$, no $(j,l)$-rung can either, and we see that $\omega \in\widetilde B^R_{j,l}$ for $J(\omega) = N$ and $L(j,\omega) = L(j)$ in \eqref{eq: decondition}.

To show the other half of \eqref{eq: symmetricdiff}, it remains to prove that
\begin{equation}\label{eq: B^R_ij minus B^R}
\lim_{j \to \infty} \limsup_{l \to \infty} \widetilde M(\widetilde B^R_{j,l} \setminus \widetilde B^R) = 0\ .
\end{equation}
We claim that if $\omega \notin \widetilde B^R$, there exists $J(\omega)$ such that for each $l \geq j \geq J(\omega)$, $\omega \notin \widetilde B^R_{j,l}$ as well. 
This implies \eqref{eq: B^R_ij minus B^R} by the same argument as before. 
Since $\omega \notin \widetilde B^R$, at least one of three defining conditions of $\widetilde B^R$ must fail. 
In each case, we will show that $\omega$ cannot be in $\widetilde B^R_{j,l}$ for all large $j$ and $l$. 
First if $g^* \notin \sigma \Delta \sigma'$ then we will never have $\omega \in B^R_{j,l}$, so we may assume the contrary. 
If $R$ is a $(j,l)$-rung for some $l \geq j \geq N$ then it connects two $(j,l)$-domain walls. As in the previous paragraph, either these $(j,l)$-domain walls are in fact distinct domain walls or they are part of the same domain wall for $j$ and $l$ large enough. This argument shows that if $R$ is not a rung, there exists $J(\omega)$ such that it will also not be a $(j,l)$-rung for $l \geq j \geq J(\omega)$. Finally, if $g^* \in \sigma \Delta \sigma'$ and $R$ is a rung, suppose that there is another rung $S$ touching $D_0(\omega)$ with energy less than the energy of $R$ minus $\e/2$. Then the same argument as above shows there exists $J'(\omega)$ such that for $l \geq j \geq J'(\omega)$, $S$ will be a $(j,l)$-rung with energy less than $E(R)-\e/2$ and therefore $\omega \notin \widetilde B^R_{j,l}$. This proves \eqref{eq: B^R_ij minus B^R} and thus \eqref{eq: symmetricdiff}.

Recall that the event $E_4^q$ is the intersection of the following: 
\begin{enumerate}
\item[a.] the three events that comprise $E_3$ defined above \eqref{eq: E3eq} (the last one of which we can replace by $\widetilde B^R$);
\item[b.] $|J_e| \leq \lambda$ for all $e^* \in S$;
\item[c.] $|J_{f_i}| \in [a_i,b_i]$ for all $i=1, \ldots, q-1$.
\end{enumerate}
 Let $E_{j,l}^q$ be the cylinder approximation of $E_4^q$ that is, the event $E_4^q$ where $\widetilde B^R$ is replaced by the cylinder event $\widetilde B_{j,l}^R$. 
Note that $E_{j,l}^q$ can be seen as an event in the translated space $T^k\Omega$ for $k$ large enough such that the box $B(0;l)$ is contained in $T^k V_H$.
Recall that in $\Omega$ as well as in $T^k\Omega$, the flexibilities $F_e$ and $F_e'$ are functions of $J$ and $\sigma,\sigma'$ given by the formula \eqref{eqn: flex}. Note also by directly applying \eqref{eq: symmetricdiff}, we find
\begin{equation}\label{eq: E_4^q minus E_jl^q}
\lim_{l \to \infty} \limsup_{j \to \infty} \widetilde M(E_4^q \Delta E_{j,l}^q) = 0\ .
\end{equation}

We claim that $E^q_{j,l}$ (and $T^k E^q_{j,l}$) has the property \eqref{eqn: event monotone}:
\begin{equation}\label{eq: Ckproperty}
\text{If }(J,\sigma,\sigma') \in E^q_{j,l} \text{ then }(J(f_q,s),\sigma,\sigma') \in E^q_{j,l} \text{ whenever }s \geq J_{f_q}\ .
\end{equation}
To check this, we first remark that if $|J_{f_i}| \in [a_i,b_i]$ for all $i=1, \ldots, q-1$ and $|J_e|\leq \lambda$ for all $e \in S$ for $(J,\sigma,\sigma')$ then this is plainly true for $(J(f_q,s),\sigma,\sigma')$ for any $s$. This handles conditions (b) and (c) of $E_{j,l}^q$. To address condition (a), we first note that the event that $g^* \in \sigma \Delta \sigma'$ (part of $\widetilde B_{j,l}^R$ in the third part of (a)) is unaffected by $J_{f_q}$, so it will continue to hold. 
In the other two parts of (a), no conditions involve the couplings except for $F_e>\varepsilon$, $F_e'>\varepsilon$. 
But since the spins at the endpoint of $f_q$ are the same, increasing $J_{f_q}$ can only possibly increase $F_e$ and $F_e'$ as seen from \eqref{eqn: flex}. (Note here that $F_e$ and $F_e'$ are simply images under $\Phi$ of $F_e(J,\sigma)$ and $F_e'(J,\sigma)$ on $\Omega$ or $T^k \Omega$, so since this argument is valid on these spaces, it holds as stated on $\Omega^*$ or $T^k \Omega^*$.)
Finally, to establish \eqref{eq: Ckproperty}, it remains to show that if $(J,\sigma,\sigma')\in \widetilde B^R_{j,l}$, then $(J(f_q,s),\sigma,\sigma')\in \widetilde B^R_{j,l}$ for $s \geq J_{f_q}$. 
Note that because $l \geq j \geq N,$ the set $D$ (defined before the statement of the present proposition) is contained in the $(j,l)$-domain wall of $g^*$ and since $f_q^*$ is adjacent to $D$, no $(j,l)$-rung containing $f^*$ can contain $f_q^*$. So increasing the value of $J_{f_q}$ to $s$ can only increase the energies of $(j,l)$-rungs that do not contain $f^*$. 
This means that if no $(j,l)$-rungs have energy less than the energy of $R$ minus $\e/2$ in $(J,\sigma,\sigma')$ then the same will be true in $(J(f_q,s),\sigma,\sigma')$ for $s \geq J_{f_q}$. 
We have thus proved \eqref{eq: Ckproperty}.

We are now in a position to use Lemma~\ref{lem: SStypemod}. 
Since $T^kM$ is just a translate of $M$, the lemma holds for the measure $T^k M$ as well, so we conclude that for all $a\in \R$ and $k \geq l \geq j \geq N$,
\begin{equation}\label{eq: floss}
T^k M(E_{j,l}^q,~ J_{f_q} \geq a) \geq  (1/2)\nu([a,\infty)) ~T^kM(E_{j,l}^q)\ .
\end{equation}
This holds trivially for $M$ replaced by $M^*$, on the space $\Omega^*$ in \eqref{eqn: Omega*}, where the flexibilities are added to the coordinates.
We would like to take limits in this inequality. For this purpose, the reader may trace through the definition of $E_{j,l}^q$ and see that this event is an intersection of a cylinder event $Y$ involving only spins and couplings
and another event $Z$ equal to $\{\widetilde F_{e_1} >\e \text{ and } \widetilde F_{e_2} > \e\}$. 
The boundary $\partial Y$ is included in the union of $\partial\{|J_e| \leq \lambda ~ \forall e^* \in S\}$, $\partial\{|J_{f_i}| \in [a_i,b_i]: \forall i=1, \ldots, q-1\}$,
$\partial\{g^* \in \sigma \Delta \sigma'\}$, $\partial\{\text{$R$ is a $(j,l)$-rung}\}$, and the boundary of the event {\it $\{$no other $(j,l)$-rung between $D_0(j,l)$ and another $(j,l)$-domain wall has energy less than the energy of $R$ minus $\e/2\}$}.
It is straightforward to see that the first four have $\widetilde M$-probability zero.
As for the fifth one, notice that the energy of a $(j,l)$-rung is a linear function of the couplings in the box $B(0;j)$ with coefficients $+1$ or $-1$.
There are only a finite number of such linear combinations. Therefore, the probability that the difference of energy between any two rungs is exactly $\e/2$ is $0$. 
By condition \eqref{eq: rigging2}, we also have $\partial Z$ of $\widetilde M$-probability zero. Therefore by the discussion preceding Remark~\ref{rem: rabbit}, we have
\[
\lim_{k \to \infty} M_k^*(E_{j,l}^q) = \widetilde M(E_{j,l}^q)\ .
\]
A similar argument holds for the left side of \eqref{eq: floss}. Averaging over $k$ and taking limits in this inequality, we find
\[
\widetilde M(E_{j,l}^q,~ J_{f_q} \geq a) \geq (1/2)\nu([a,\infty)) ~\widetilde M(E_{j,l}^q)\ ,
\]
Now we take $l \to \infty$ and $j \to \infty$, using \eqref{eq: E_4^q minus E_jl^q} to obtain
\[
\widetilde M(E_4^q,~ J_{f_q} \geq a) \geq (1/2)\nu([a,\infty)) ~\widetilde M(E_4^q)\ .
\]
By the induction hypothesis, $\widetilde M(E_4^q) >0$.
To finish the proof of the lemma, it thus suffices to take $a=a_q = 4b_{q-1} +1$ and choose any $b_q>a_q$.
\end{proof}

\subsubsection{Finishing the proof}
In this subsection we use Proposition~\ref{prop: xe} to prove a final proposition about rung energies. 
This will allow us to reach a contradiction and establish \eqref{eqn: 2}.

Recall that $f$ refers to the fixed edge connecting $(0,1)$ to $(1,1)$ and $g$ is the edge connecting the origin to $(1,0)$, see Figure \ref{fig: magic_rung}. Our goal in this section is to show that $J_f$ can be modified so that the energy of some rung that contains $f^*$ decreases below the energies of all rungs that do not contain $f^*$. To do this, we introduce two variants of $I(\omega)$, dealing with rungs that contain $f^*$ and rungs that do not.

On the event $X_\e$, we define the variable $I'(\omega)$ to be the infimum of energies of all rungs that touch $D_0(\omega)$ (the domain wall that contains $h^*=g^*$) and that do not contain $f^*$. 
Also we define $\widetilde I(\omega)$ to be the infimum of energies of all rungs that contain $f^*$. Later in the proof we will use a small technical fact: the distribution of $\widetilde I(\omega) - I'(\omega)$ (under $\widetilde M$) can have only countably many point masses. Therefore we may choose $\e$ small enough so that the conclusion of Proposition~\ref{prop: xe} holds and so that
\begin{equation}\label{eq: technical}
\widetilde M(\omega: \widetilde I(\omega) - I'(\omega) = \e/2 \text{ or } \widetilde I(\omega) - I'(\omega) = -\e/4) = 0\ .
\end{equation}
This $\e$ will be fixed for the rest of the paper.

Let $Y_\e$ be the event that:
\begin{enumerate}
\item $\sigma\Delta\sigma'$ is disconnected and $I>0$;
\item $g^* \in \sigma \Delta \sigma'$;
\item $f^*$ is in a rung $R$ that satisfies $E(R) <I'(\omega)-\e/4$.
\end{enumerate}

The next two propositions establish the desired contradiction.
The idea is that, on the one hand (cf. Proposition~\ref{prop: ye=0}), $Y_\e$ must have zero probability since by
Lemma~\ref{lem: infinitelymanyedges} and Remark~\ref{rem: rem1} an event along the domain wall occurs infinitely often, whereas $f^*$ must be unique along the domain wall by the definition of $I'$.
On the other hand, we will use Proposition~\ref{prop: xe} in Proposition~\ref{prop: ye} to show that the event $Y_\varepsilon$ must have positive probability.

\begin{prop}\label{prop: ye=0}
The following statement holds.
\[
\widetilde M(Y_\e)=0\ .
\]
\end{prop}

\begin{prop}\label{prop: ye}
If $\widetilde M(X_\e \cap \{\widetilde F_f>\e\})>0$, then
\[
\widetilde M(Y_\e)>0\ .
\]
\end{prop}

\begin{proof}[Proof of Proposition~\ref{prop: ye=0}]
For a dual vertex $b^*$, let $Y_\e(b) \subseteq \widetilde \Omega$ be the event that
\begin{enumerate}
\item $\sigma\Delta\sigma'$ is disconnected and $I>0$;
\item $b^* \in \sigma \Delta \sigma'$;
\item there is a dual edge $e^*$, sharing a dual endpoint with $b^*$, that is the first edge of a rung $R$ with $E(R) <I'_{b,e}(\omega)-\e/4$.
Here $I'_{b,e}$ is the infimum of energies of the rungs not containing $e^*$ and touching the domain wall of $b^*$.
\end{enumerate}
In this notation, the $Y_\e$ corresponds to the case $b=g$ and $e=f$.
By definition of $I'_{b,e}$, for each domain wall $D$, there are at most two dual edges $b^*$ such that $Y_\e(b)$ occurs (one for each side of $D$).
By the same argument as in Lemma~\ref{lem: infinitelymanyedges} (with $B_e(\e,K)$ replaced by $Y_\e(b)$), it follows
that $\widetilde M(Y_\e(b))=0$ for all dual edges $b$ (see Remark~\ref{rem: rem1}), so $\widetilde M(Y_\e)=0$.
\end{proof}

\begin{proof}[Proof of Proposition~\ref{prop: ye}]
On the event $X_\e \cap \{\widetilde F_f>\e\}$, either the spins at the endpoints of $f$ are the same or they are different (in both $\sigma$ and $\sigma'$). 
Let us suppose that:
\[
\widetilde M(X_\e,~\widetilde F_f>\e,~\sigma_f =\sigma_f' =+1)>0\ .
\]
The subsequent argument can easily be modified in the case $\sigma_f=\sigma'_f=-1$ (using an obvious analogue of Lemma~\ref{lem: backmodify}.)
Define $\widetilde C_f = \max\{C_f, C_f'\}$. We may choose $a \in \mathbb{R}$ such that 
\begin{equation}\label{eq: toaster1}
\widetilde M(X_\e,~\widetilde F_f>\e,~\sigma_f=\sigma'_f=+1,~\widetilde C_f \in (a,a+\e/8))>0\ 
\end{equation}
and because the distribution of $\widetilde C_f$ can have countably many point masses, we may further restrict our choice of $a$ so that 
\begin{equation}\label{eq: rigging}
\widetilde M(\widetilde C_f = a \text{ or } a+\e/8) = 0\ .
\end{equation}
By property \eqref{eqn: equiv}, for each $k$,
\[
T^k M((J,\sigma,\sigma')~:~\sigma_f= \sigma'_f =+1,~J_f < \max\{C_f(J,\sigma),C_f(J,\sigma')\}) = 0\ .
\]
This is an open cylinder event in $\Omega^*$, thus after averaging and taking liminf,
\begin{equation}\label{eq: toaster2}
\widetilde M(\sigma_f=\sigma'_f=+1,~J_f < \widetilde C_f) \leq \liminf_{k \to \infty} M_k^*(\sigma_f=\sigma'_f=+1,~J_f < \widetilde C_f) = 0\ .
\end{equation}
If $J_f\geq\widetilde C_f$ and $\widetilde F_f=\max\{|J_f-C_f|,|J_f-C_f'|\}>\e$, then $J_f>C_f+\e$ and $J_f'>C_f'+\e$. 
By combining \eqref{eq: toaster1} and \eqref{eq: toaster2}, we thus find
\[
\widetilde M(X_\e,~\widetilde C_f \in (a,a+\e/8), J_f \geq a+\e)>0\ .
\]

Recall that $\widetilde I(\omega)$ is the infimum of energies of all rungs that contain $f^*$. On the event $X_\e$, we have $\widetilde I(\omega)< I'(\omega)+\e/2$.
Therefore if $\widetilde B$ is the event that
\begin{enumerate}
\item $g^* \in \sigma \Delta \sigma'$ but $f^* \notin \sigma \Delta \sigma'$;
\item  $\widetilde I(\omega)<I'(\omega)+\e/2$,
 \end{enumerate}
 then
\begin{equation}\label{eq: tildeB4}
\widetilde M(\widetilde B,~ \widetilde C_f \in (a,a+\e/8),~J_f \geq a+\e)>0\ .
\end{equation}
Note that condition 2 of $\widetilde B$ only makes sense if $f^*$ is actually in a rung; however, in the support of $\widetilde M$, $\sigma$ and $\sigma'$ are ground states, so their interface does not contain loops. Thus when condition 1 of $\widetilde B$ holds and $\omega$ is in the support of $\widetilde M$, $f^*$ is in a rung.

From this point on, the strategy is similar to the proof of Lemma~\ref{lem: almostSS}. 
The idea is to use Lemma~\ref{lem: backmodify} to lower $\widetilde I(\omega)$ below $I'(\omega)-\e/4$.
Let $\widetilde P$ be the event $\widetilde B$ with the condition $\widetilde I(\omega)<I'(\omega)+\e/2$ replaced by
$\widetilde I(\omega)<I'(\omega)-\e/4$. We will show that 
$$\widetilde M(\widetilde P)>0\ .$$
A quick look at  \eqref{eq: tildeB4} can convince us that this is possible since $J_f$ could be lowered by $3\e/4$
and still not reach the critical value. Since $\widetilde I(\omega)$ depends linearly on $J_f$ by definition, it will be itself lowered by $3\e/4$
and become lower than $I'(\omega)$ by $\e/4$.
To make this reasoning rigorous, as in the proof of Lemma~\ref{lem: almostSS}, we must bring the problem back to the half-plane measure $M$ and find
a cylinder approximation for both $\widetilde B$ and $\widetilde P$.

Let $B(0;j)$ be the box of side-length $j$ centered at $0$ and let $l \geq j$.
Recall the definitions of  $(j,l)$-domain walls and $(j,l)$-rungs below \eqref{eq: symmetricdiff}.
Let $D_0(j,l), D_1(j,l), \ldots, D_t(j,l)$ be the $(j,l)$-domain walls in $B(0;j)$ and $D_0(j,l)$ be the one containing $g^*$ (if it exists).
For $l \geq j$ and $\omega \in \widetilde \Omega$ such that $g^* \in \sigma \Delta \sigma'$, write $I_{j,l}'(\omega)$ (the cylinder approximation of $I'(\omega)$) as the infimum of all energies of $(j,l)$-rungs which touch $D_0(j,l)$ but do not contain the dual edge $f^*$. Write $\widetilde I_{j,l}(\omega)$ (the cylinder approximation of $\widetilde I(\omega)$) for the infimum of all energies of $(j,l)$-rungs which contain $f^*$. 
Let $\widetilde B_{j,l}\subseteq \widetilde \Omega$ be the cylinder approximation of $\widetilde B$:
\begin{enumerate}
\item $g^* \in \sigma \Delta \sigma'$ but $f^* \notin \sigma \Delta \sigma'$.
\item $\widetilde I_{j,l}(\omega)<I'_{j,l}(\omega)+\e/2$. 
\end{enumerate}
We define the cylinder approximation $\widetilde P_{j,l}$ of $\widetilde P$ similarly with $\e/2$ replaced by $-\e/4$. 
There may be no $(j,l)$ rungs, but their existence is implicit in condition 2 (in other words, it is implied in condition 2 that the variables $\widetilde I_{j,l}(\omega)$ and $I'_{j,l}(\omega)$ are defined). We claim that
\begin{equation}\label{eq: tildeB3}
\begin{aligned}
&\lim_{j \to \infty} \limsup_{l \to \infty} \widetilde M(\widetilde B_{j,l} \Delta \widetilde B) = 0\\
&\lim_{j \to \infty} \limsup_{l \to \infty} \widetilde M( P_{j,l} \Delta P_\e) = 0\ .
\end{aligned}
\end{equation}
We give the proof for $\widetilde{B}$. The proof for $\widetilde P$ is identical with $\e/2$ replaced by $-\e/4$.

To begin with, let $\omega$ be a configuration such that $g^* \in \sigma \Delta \sigma'$ and $f \notin \sigma \Delta \sigma'$ (this is true for all configurations in $\widetilde B$ or in $\widetilde B_{j,l}$). Note that for fixed $j$,
\[
\widetilde I_j(\omega):= \lim_{l \to \infty} \widetilde I_{j,l}(\omega) \text{ exists }
\]
and equals the infimum of energies of all rungs that stay in $B(0;j)$ and contain $f^*$. Clearly,
\[
\lim_{j \to \infty} \widetilde I_j(\omega) = \widetilde I(\omega)\ .
\]
The analogous statements are true for $I'(\omega)$ (defining $I'_j(\omega)$ similarly). Therefore given $\delta>0$ we may choose $J(\omega)$ such that $j \geq J(\omega)$ implies that
\[
|\widetilde I_j(\omega) - \widetilde I(\omega)|<\delta/2 \text{ and } |I'_j(\omega) - I'(\omega)|<\delta/2\ .
\]
For any such $j$ we can find $L(j,\omega)$ such that for $l \geq L(j,\omega)$,
\[
|\widetilde I_{j,l}(\omega) - \widetilde I_j(\omega)| < \delta/2 \text{ and } |I'_{j,l}(\omega)-I'_j(\omega)|<\delta/2\ .
\]
Therefore for $j \geq J(\omega)$ and $l \geq L(j,\omega)$,
\begin{equation}\label{eq: hamster}
|\widetilde I_{j,l}(\omega) - \widetilde I(\omega)|<\delta \text{ and } |I'_{j,l}(\omega)-I'(\omega)|<\delta\ .
\end{equation}
We first show that 
\begin{equation}
\label{eqn: B minus B_ij}
\lim_{j \to \infty} \limsup_{l \to \infty} \widetilde M(\widetilde B \setminus \widetilde B_{j,l}) = 0\ .
\end{equation}
Suppose that $\omega \in \widetilde B$. 
Then $\widetilde I(\omega)< I'(\omega)+\e/2$ and, combining this with \eqref{eq: hamster}, we may choose $\delta = \delta(\omega)$ so small that for $j \geq J(\omega)$ and $l \geq L(j,\omega)$,
\[
\widetilde I_{j,l}(\omega)<I'_{j,l}(\omega) + \e/2\ .
\]
Because $\omega \in \widetilde B$, the first condition of $\widetilde B_{j,l}$ holds directly.
Equation \eqref{eqn: B minus B_ij} follows from this using the same reasoning as for \eqref{eq: stapler1}.

We now prove that
\begin{equation}\label{eq: hamster2}
\lim_{j \to \infty} \limsup_{l \to \infty} \widetilde M(\widetilde B_{j,l} \setminus \widetilde B) = 0\ .
\end{equation}
As before, we need to show that if $\widetilde I(\omega)\geq I'(\omega)+\e/2$ then there is $J(\omega)$ such that for each $j \geq J(\omega)$, there is an $L(j,\omega)$ such that if $l \geq L(j,\omega)$ then $\widetilde I_{j,l}(\omega) \geq I_{j,l}'(\omega)$. If $\widetilde I(\omega)>I'(\omega)+\e/2$ then the arguments leading up to \eqref{eq: hamster} prove this immediately. In the other case, let $\widetilde U$ be the event that $\widetilde I(\omega) = I'(\omega)+\e/2$. This event has $\widetilde M$-probability zero by \eqref{eq: technical} (for the approximation of $\widetilde P$,  one has $\widetilde I(\omega) = I'(\omega)-\e/4$ instead).
So
\[
\lim_{j \to \infty} \limsup_{l \to \infty} \widetilde M( (\widetilde B_{j,l} \cap \widetilde U^c) \setminus (\widetilde B \cap \widetilde U^c)) = 0\ .
\]
However, $\widetilde U$ has $\widetilde M$-probability zero, so this proves \eqref{eq: hamster2}.

Notice that 
\[
\widetilde B_{j,l} \cap \{\widetilde C_f \in (a,a+\e/8) \} \cap \{J_f \geq a+\e\}
\]
is a cylinder event in $\widetilde \Omega$.
This event also makes sense under the measure $T^kM$ on the half-plane for
$\widetilde C_f= \max \{C_f(J,\sigma),C_f(J,\sigma')\}$ (where the critical values are functions
as defined in \eqref{eqn: c_e}) and for $k\geq l\geq j$ so that the boxes are contained in $T^k V_H$.
We now analyze the probability $T^kM(\widetilde B_{j,l},~C_f\in (a,a+\e/8),~J_f \geq a+\e)$.
Let $\widetilde K_{j,l}(\omega):= \widetilde I_{j,l}(\omega)-J_f$ be the infimum of the energies of $(j,l)$-rungs
where the contribution from the edge $f$ is removed. 
If $\widetilde I_{j,l}(\omega)<I'_{j,l}(\omega)+\e/2$ and $J_f\geq a+\e$, then
\begin{equation}
\label{eqn: K}
\widetilde K_{j,l}(\omega)= \widetilde I_{j,l}(\omega)-J_f < I'_{j,l}(\omega)-a - \e/2\ .
\end{equation}
Define $A_{j,l} \subseteq T^k\Omega$ as the intersection of the following events. 
\begin{enumerate}
\item $f^* \notin \sigma \Delta \sigma'$ and $\sigma_f=\sigma'_f=+1$. 
\item $g^* \in \sigma \Delta \sigma'$ and $\widetilde K_{j,l}(\omega) < I'_{j,l}(\omega)-a - \e/2$.
\item $\widetilde C_f \in (a,a+\e/8)$. 
\item $\sigma, \sigma'$ are in $\G(J)$, the ground states in $\Z\times\N$.
\end{enumerate}
Implicit in the second condition is that the variables $\widetilde K_{j,l}(\omega)$ and $I'_{j,l}(\omega)$ are actually defined; in particular, $f^*$ must be in some $(j,l)$-rung. Although the last condition does not give a cylinder event, it will be used to apply Lemma~\ref{lem: backmodify}. $A_{j,l}$ is an intermediary event between $\widetilde B_{j,l}$ and $P_{j,l}$. 
On the set $\{\sigma,\sigma' \in \G(J)\}$, $\widetilde B_{j,l} \cap \{\widetilde C_f \in (a,a+\e/8) \} \cap \{J_f \geq b\}$ implies $A_{j,l}$ by \eqref{eqn: K}, so
\begin{equation}\label{eq: Btilde2}
T^kM(\widetilde B_{j,l},~\widetilde C_f \in (a,a+\e/8),~J_f \geq a+\e) \leq T^kM(A_{j,l})\ .
\end{equation}

We claim that $A_{j,l}$ (and $T^kA_{j,l}$) has the property \eqref{eqn: event monotone 2} of Lemma~\ref{lem: backmodify}:
$$
\text{If $(J,\sigma,\sigma') \in A_{j,l}$ and $J_f \geq a+\e/8$, then $(J(f,s),\sigma,\sigma') \in A_{j,l}$ for all $s \geq a+\e/8$}\ . 
$$
To verify this, note that the defining condition 1 of $A_{j,l}$ does not depend on $J_f$, so if $(J,\sigma,\sigma')$ satisfies it, so will $(J(f,s),\sigma,\sigma')$ for all $s$. 
Next we argue that $\sigma,\sigma' \in \G(J(f,s))$ for all $s \geq a+\e/8$. 
This holds because $\sigma_f=\sigma'_f=+1$, $J_f \geq a+\e/8 >\widetilde C_f=\max\{C_f(J,\sigma),C_f(J,\sigma')\}$.
Clearly condition 3 holds for $(J(f,s),\sigma,\sigma')$ as the critical values do not depend on the coupling at $f$. 
Last, because $\sigma_f=\sigma'_f=+1$, we see that $\widetilde K_{j,l}(J(f,s),\sigma,\sigma')$ does not depend on $s$ since the contribution of $J_f$ to $\widetilde I_{j,l}$ is removed. 
Also the variable $I_{j,l}'$ does not depend on $J_f$ by construction. Therefore condition 4 holds for $(J(f,s),\sigma,\sigma')$. 

We are now in the position to apply Lemma~\ref{lem: backmodify}. 
Because $A_{j,l}$ satisfies the hypotheses of the lemma for $c=a+\e/8$, we select $d=a+\e/4$ and find
\[
T^kM(A_{j,l},~J_f \in [a+\e/8,a+\e/4]) \geq \nu([a+\e/8,a+\e/4]) ~T^kM(A_{j,l})\ .
\]
When $A_{j,l}$ occurs and $J_f \leq a+\e/4$,
\begin{eqnarray*}
\widetilde I_{j,l}(\omega) = \widetilde K_{j,l}(\omega) + J_f &<& I'_{j,l}(\omega) -a -\e/2 + a +\e/4 \\
&=& I'_{j,l}(\omega)-\e/4\ .
\end{eqnarray*}
Therefore, writing $r = \nu([a+\e/8,a+\e/4])$, 
\[
T^k M(A_{j,l},~\widetilde I_{j,l}(\omega) \leq I'_{j,l}(\omega)-\e/4) \geq r~T^kM(A_{j,l})\ ,
\]
and by \eqref{eq: Btilde2},
\begin{equation}\label{eq: Btilde3}
T^k M(A_{j,l},~\widetilde I_{j,l}(\omega) \leq I'_{j,l}(\omega)-\e/4) \geq r~T^kM(\widetilde B_{j,l}, \widetilde C_f \in (a,a+\e/8), J_f \geq a+\e)\ .
\end{equation}
Now $A_{j,l}\cap \{\widetilde I_{j,l}(\omega) \leq I'_{j,l}(\omega)-\e/4\}$ is contained in $\widetilde P_{j,l}$.  By \eqref{eq: Btilde3},
\begin{equation}\label{eq: lastreally1}
T^kM( \widetilde P_{j,l}) \geq r~T^kM(\widetilde B_{j,l}, \widetilde C_f \in (a,a+\e/8),J_f \geq a+\e)\ .
\end{equation}

We now want to average over $k$ and take the limit in \eqref{eq: lastreally1}. First note that $\widetilde P_{j,l}$ is an event that only involves spins and couplings. 
Furthermore, the only non-trivial contribution to the boundary $\partial \widetilde P_{j,l}$ is $\partial \{\widetilde I_{j,l}(\omega) \leq I'_{j,l}(\omega)-\e/4\}$. 
This event is contained in the event that there are two distinct $(j,l)$-rungs in the box $B(0;N)$ whose energies differ by exactly $\e/4$. Since the energy is a linear function of the couplings and of the spins, and since there are only a finite
number of possible rungs in $B(0;N)$, this event has $\widetilde M$-probability zero by the continuity of $\nu$. 
Thus by the discussion preceding Remark~\ref{rem: rabbit} we may take the limit on the left to get
\[
\lim_{k \to \infty} M_k^*(P_{j,l}) = \widetilde M(P_{j,l})\ .
\]
By \eqref{eq: rigging}, and reasoning similar to above, the boundary of the event on the right side of \eqref{eq: lastreally1} also has $\widetilde M$-probability zero. Therefore we can average over $k$ in \eqref{eq: lastreally1} and take the limit to finally get
\begin{equation}\label{eq: finaleq11}
\widetilde M( P_{j,l}) \geq r~\widetilde M(\widetilde B_{j,l},\widetilde C_f \in (a,a+\e/8),J_f \geq a+\e)\ .
\end{equation}
Finally, it suffices to take $l \to \infty$ and $j \to \infty$. By \eqref{eq: tildeB3}, the right side converges to
\[
r\widetilde M(\widetilde B, \widetilde C_f \in (a,a+\e/8),J_f \geq a+\e) >0\ .
\]
The probability is positive by \eqref{eq: tildeB4}.
The left side of \eqref{eq: finaleq11} converges to $\widetilde M (\widetilde P)$ again by \eqref{eq: tildeB3}.
Thus $\widetilde M(P_\e) >0$. As $P_\e \cap \{I>0\} \subset Y_\e$ and $\widetilde M(I=0)=0$, this completes the proof.
\end{proof}

\end{document}